%% file: BayesFreqCI.tex
\documentclass[aos,preprint]{imsart}
\RequirePackage[OT1]{fontenc}
\RequirePackage{amsthm,amsmath}
\RequirePackage[numbers]{natbib}
\usepackage{calrsfs}
\usepackage{pictexwd}
\usepackage{bm,amssymb,graphicx,url}
\usepackage{color}
\usepackage{epstopdf}
\usepackage[normalem]{ulem}

\usepackage[colorlinks,citecolor=blue,urlcolor=blue]{hyperref}

\startlocaldefs

\input commands3.tex

\numberwithin{equation}{section}
\theoremstyle{plain}
\endlocaldefs

\usepackage{amsmath,here,amsfonts,latexsym,graphicx,here}
\usepackage{amsthm,curves}
\usepackage{graphicx}
\usepackage{caption}
\usepackage{subcaption}
\usepackage{float}

\def\one{{\mathbf 1}}

\def\th{\theta}

\def\P{{\mathbb P}}

\def\G{{\mathbb G}}

\def\th{\theta}

\def\P{{\mathbb P}}

\def\G{{\mathbb G}}

\def\eps{\epsilon}
\def\s{\sigma}

\newtheorem{theorem}{Theorem}[section]
\newtheorem{lemma}{Lemma}[section]

\newtheorem{remark}{Remark}[section]
\newtheorem{corollary}{Corollary}[section]

\begin{document}
\begin{frontmatter}
\title{Credible intervals and bootstrap confidence intervals in monotone regression}
\runtitle{monotone regression}
\begin{aug}
\author{\fnms{Piet}~\snm{Groeneboom} and \fnms{Geurt}~\snm{Jongbloed}}
\address{Delft Institute of Applied Mathematics, Delft University of Technology\\ \rm
P.Groeneboom@tudelft.nl, G.Jongbloed@tudelft.nl}
\end{aug}
\date{\today}

\begin{abstract}
In the recent paper \cite{moumita:21}, a Bayesian approach for constructing confidence intervals in monotone regression problems is proposed, based on credible intervals. We view this method from a frequentist point of view, and show that it corresponds to a percentile bootstrap method of which we give two versions. It is shown that a (non-percentile) smoothed bootstrap  method has better behavior and does not need correction for over- or undercoverage. The proofs use martingale methods.
\end{abstract} 



\end{frontmatter}
\section{Introduction}
\label{sec:Intro}
In many fields of application, inference on a monotone function is both natural and needed. Many  examples of applications can be found in the books on shape constrained statistical inference \cite{b4:72}, \cite{rwd:88}, \cite{silvapulle2005constrained} and \cite{piet_geurt:14}. Natural estimators of monotone functions exist. In order to assess the accuracy of these, there is a need for uncertainty quantification. In the frequentist setting, this can be done by (pointwise) confidence intervals for the monotone function of interest. From a Bayesian perspective, credible intervals are a common method to quantify uncertainty.

 We consider the following model for our observations $D_n=\{(X_1,Y_1),\ldots,(X_n,Y_n)\}$, also considered in \cite{moumita:21}.
 \begin{align*}
 	Y_i=f_0(X_i)+\eps_i,\qquad i=1,\dots,n.
 \end{align*}
 Here $f_0:[0,1]\mapsto\R$ is monotone nondecreasing, the $\e_i$ are i.i.d.\ sub-Gaussian with expectation $0$ and variance $\s_0^2$, independent of the $X_i$'s, and the $X_i$ are i.i.d. with non-vanishing density $g$ on $[0,1]$. The classical least squares estimator (LSE) $\hat f_n$ of $f_0$, under the condition that $\hat f_n$ is nondecreasing minimizes
 \begin{align}\label{eq:sumsquares}
 	\sum_{i=1}^n\left(Y_i-f(X_i)\right)^2
 \end{align}
 over all nondecreasing functions $f:[0,1]\mapsto\R$. This estimator, which is also the maximum likelihood estimator if we assume the $\eps_i$ to be i.i.d.\ centered normally distributed, can be explicitly constructed based on the data. Denoting by $0<x_{1}<x_{2}<\cdots<x_{n}<1$ the observed ordered $X_i$'s and by $y_i$ the corresponding $Y$ values (so $(x_1,y_1),\ldots,(x_n,y_n)$ represent the original data pairs, but with first coordinate sorted in increasing order), Lemma 2.1 in \cite{piet_geurt:14} shows that $\hat{f}_n$ can be taken piecewise constant on the intervals $(x_{i-1},x_{i}]$, $1\le i\le n$ and that $\hat{f}_n(x_{i})$ can be obtained as the left derivative of the greatest convex minorant of the diagram consisting of the points $\{P_i\,:\,0\le i\le n\}$ with
 $$
 P_0=(0,0) \mbox{ and for } 1\le i\le n,\,\,\,P_i=\left(\frac{i}{n},\frac1n\sum_{j=1}^iy_j\right),
 $$
 evaluated at the point $P_i$. 
 
 As is clear from this characterization, $\hat{f}_n$ will be a nondecreasing step function with its jumps concentrated on a data-dependent subset of the observed points $x_{i}$. 
 Fixing a number of points in $[0,1]$, say $0=\tau_0<\tau_1<\tau_2<\cdots<\tau_m=1$, one can also minimize (\ref{eq:sumsquares}) over all nondecreasing functions, piecewise constant on the intervals $I_j=(\tau_{j-1},\tau_j]$. Then, writing  $n_j=|\{i\,:\,i\in I_j\}|$ and $\bar{y}_j=(\sum_{x_i\in I_j}y_i)/n_j$, we have
 \begin{align}\nonumber
 	\sum_{i=1}^n\left(y_i-f(x_i)\right)^2&=\sum_{j=1}^m\sum_{x_i\in I_j}\left(y_i-\bar{y}_j+\bar{y}_j-f(\tau_j)\right)^2=\\ \label{eq:decompss}
 	&=\sum_{j=1}^m\sum_{x_i\in I_j}\left(y_i-\bar{y}_j\right)^2+\sum_{j=1}^m\left(\bar{y}_j-f(\tau_j)\right)^2n_j,
\end{align}
where we use that for the current function class, $f(x_i)=f(\tau_j)$ for $x_i\in I_j$. As the first  term in (\ref{eq:decompss}) does not involve $f$, minimizing it boils down to a weighted isotonic regression. The solution to this minimization problem also allows for a graphical construction. The optimal value of $f(\tau_j)$ is the left derivative, taken at the point $P_j$, of the greatest convex minorant of the diagram of points consisting of
\begin{align*}
	P_0=(0,0), \,\, P_j=\left(\frac1n\sum_{k=1}^jn_k,\frac1n\sum_{k=1}^j n_k\bar{y}_k \right),\,\,\,1\le j\le m.
\end{align*} 
From decomposition (\ref{eq:decompss}) it also follows that the piecewise constant function $f$ defined by $f(x_i)=\bar{y}_j$ for $x_i\in I_j$ is the least squares estimator over the class of piecewise constant functions without imposing the restriction of monotonicity.

In \cite{moumita:21}, a Bayesian method is proposed for constructing pointwise confidence intervals for a monotone regression function, based on credible intervals. The method is proved to give overcoverage for large sample sizes, but a correction table is given in \cite{moumita:21} to correct for the overcoverage. Purely based on the algorithm that results in the credible intervals, the approach can be seen as a particular percentile bootstrap method. 

In Section  \ref{section:credible_intervals} we describe the approach in \cite{moumita:21} to construct confidence intervals via credible intervals. 
In Section \ref{sec:credboot} we give the interpretation of the credible intervals as percentile bootstrap intervals and in particular Theorem \ref{theorem:limit_posterior} for the bootstrap procedure, corresponding to the key Theorem 3.3 in \cite{moumita:21} for the construction of the credible intervals. In proving Theorem \ref{theorem:limit_posterior} we use a martingale method.

In analogy with the Bayesian procedure, we construct the bootstrap intervals by generating normal noise variables (following \cite{moumita:21}), using the empirical Bayes method for estimating the variance of these variables, defined in Section \ref{sec:credboot}.
In subsection \ref{subsec:classic_bootstrap} we define a classical bootstrap procedure, where we resample with replacement from the original data, and do not have to estimate the variance. These two methods correspond, respectively,  to the ``regression method'' (holding the regressors $X_i$ in the regression model fixed), and the ``correlation model'' (where we consider the $X_i$ as random) in the terminology of \cite{hall:92book}. The results of the three methods are highly similar.
 
It has been proved by several authors that  the straightforward bootstrap is inconsistent in this situation (see, e.g., \cite{kosorok:08}, \cite{SenXu2015} and \cite{sen_mouli_woodroofe:10} for results related to this phenomenon).This straightforward bootstrap uses resampling with replacement from the pairs $(X_i,Y_i)$ and computes the monotone least squares estimator $\hat f_n^*$ based bootstrap samples and approximates the distribution of $n^{1/3}\left(\hat{f}_n(t_0)-f_0(t_0)\right)$ by that of the analogous `bootstrap quantity' $n^{1/3}\left(\hat{f}_n^*(t_0)-\hat{f}_n(t_0)\right)$. 
The Bayesian approach and the percentile bootstrap approach circumvent this difficulty by using the convergence in distribution of the random variable (as a function of $D_n=\{(X_1,Y_n),\dots,(X_n,Y_n)$\})
\begin{align}
\label{convergence_conditional_prob}
\P\left(n^{1/3}\bigl\{\hat f_n^*(t_0)-f_0(t_0)\bigr\}\le x\Bigm|D_n\right)
\end{align}
to
\begin{align}
\label{conditional_limit}
\P\left[\left(\frac{4\s_0^2f_0'(t_0)}{g(t_0)}\right)^{1/3}\text{\rm argmin}_{t\in\R}\left[W_1(t)+W_2(t)+t^2\right]\le x\Bigm| W_1\right]
\end{align}
see Theorem 3.3 in \cite{moumita:21} and Theorem \ref{theorem:limit_posterior}, where $W_1$ and $W_2$ are two independent standard two-sided Brownian motions, originating from zero. Here $\hat f_n^*$ is either the projection-posterior Bayes estimate (to be described in Section \ref{section:credible_intervals}), in which case we would write
\begin{align*}
\Pi\left(n^{1/3}\bigl\{\hat f_n^*(t_0)-f_0(t_0)\bigr\}\le x\Bigm|D_n\right)
\end{align*}
instead of (\ref{convergence_conditional_prob}), or the percentile bootstrap estimate $\hat f_n^*$.
The limit (\ref{conditional_limit}) leads to credible intervals which asymptotically give overcoverage, which can be corrected for as described in \cite{moumita:21}.

In Section \ref{sec:CI_cuberoot_n} an altogether different method for constructing the confidence intervals is given, where we use the smoothed (non-percentile) bootstrap. Here we keep the regressors fixed again, and resample with replacement residuals w.r.t.\ a smooth estimate of the regression function: the Smoothed Least Squares Estimator (SLSE).
In this way, using theory from \cite{geurt_piet:23}, consistent confidence intervals are constructed.

In fact, instead of $n^{1/3}\bigl\{\hat f_n^*(t_0)-f_0(t_0)\bigr\}=n^{1/3}\bigl\{\hat f_n^*(t_0)-\hat f_n(t_0)+\hat f_n(t_0)-f_0(t_0)\bigr\}$
we can now consider
\begin{align*}
n^{1/3}\bigl\{\hat f_n^*(t_0)-\tilde f_{nh}(t_0)+\hat f_n(t_0)-f_0(t_0)\bigr\},
\end{align*}
where $\hat f_n^*$ is based on  sampling with replacement from the residuals w.r.t.\ the SLSE $\tilde f_{nh}$ with bandwidth $h$ of order $n^{-1/5}$. In contrast with Theorem 3.3 in \cite{moumita:21} and Theorem \ref{theorem:limit_posterior} in the present paper,
we now have convergence of
\begin{align*}
\P\left(n^{1/3}\bigl\{\hat f_n^*(t_0)-\tilde f_{nh}(t_0)+\hat f_n(t_0)-f_0(t_0)\bigr\}\le 0\Bigm|D_n\right)
\end{align*}
to the uniform distribution on $[0,1]$ (using the symmetry of the limit distribution of $n^{1/3}\{\hat f_n(t_0)-f_0(t_0)\}$), see Theorem \ref{theorem:limit_smooth_bootstrap}.

For the non-percentile bootstrap, however, it is more natural to consider
\begin{align*}
n^{1/3}\bigl\{\hat f_n(t_0)-\{\hat f_n^*(t_0)-\tilde f_{nh}(t_0)\}-f_0(t_0)\bigr\}
\end{align*}
(avoiding ``looking up the wrong tables, backwards'', see the discussion on p.\ 938 of \cite{hall:88}), for which we also get convergence to the the uniform distribution of
\begin{align*}
\P\left(n^{1/3}\bigl\{\hat f_n(t_0)-\{\hat f_n^*(t_0)-\tilde f_{nh}(t_0)\}-f_0(t_0)\bigr\}\le 0\Bigm|D_n\right),
\end{align*}
implying the consistency of the smoothed bootstrap method. This method of constructing confidence intervals seems superior in comparison to the Bayesian method and the percentile bootstrap intervals, as is suggested by our simulations of the coverage of the different methods.

 \section{Credible intervals}
 \label{section:credible_intervals}
 In \cite{moumita:21}, a Bayesian approach to construct confidence intervals for a monotone regression function is proposed. A prior distribution is defined on the class of functions on $[0,1]$, supported on a sieve of piecewise constant functions. More specifically, the interval $[0,1]$ is partitioned into $J$ intervals $I_j=((j-1)/J,j/J]$, $1\le j\le J$. In the notation of the previous section, $\tau_j=j/n$. A draw from the prior distribution is then represented by
 \begin{align}
 	\label{eq:reprfunctitotheta}
 	f_{\bm\th}=\sum_{j=1}^J \th_j 1_{I_j},\qquad \bm\th=(\th_1,\dots,\th_J),
 \end{align}
 where the $\th_j$ are independent normal random variables with expectation $\zeta_j$ and variance $\sigma_0^2\lambda_j^2$, where $0<\l_j<\infty$ (including noise variance  $\sigma_0^2$ as a factor is only done for convenience in formulas to follow). Note that function (\ref{eq:reprfunctitotheta}) will not automatically be monotone, a requirement that would seem natural in this setting. The main reason not to impose this, is that with this prior distribution, the posterior distribution can be conveniently analytically computed. Indeed, as seen in the Appendix, the posterior distribution of $\bm\th$ has independent coordinates, where $\th_j$ has distribution
 \begin{align}
 	\label{posterior_theta}
 	\th_j\sim N\left(\frac{n_j\bar y_j+\zeta_j/\l_j^2}{n_j+1/\l_j^2},\frac{\s_0^2}{n_j+1/\l_j^2}\right),\qquad n_j=\sum_{i=1}^n 1_{I_j}(x_i).
 \end{align}
Here, as before, $\bar y_j$ is the mean of the $y_i$ for the $x_i$ belonging to the $j$-th interval $I_j$. As mentioned in the previous section, this corresponds to the MLE of $f$ over the (nonrestricted) class of functions which are constant on the intervals $I_j$.  A draw from the posterior on the set of piecewise constant functions on $[0,1]$ proceeds via (\ref{eq:reprfunctitotheta}), based on a draw from the posterior of ${\bm\th}$. The resulting function will in general not be monotone, so the support of the posterior extends outside the set of monotone functions on $[0,1]$. 
 
 Following ideas of \cite{lin2014bayesian}, \cite{bhaumik2015bayesian} and \cite{bhaumik2017efficient}, in \cite{moumita:21}  a draw $f_{\bm\th}$ from the `raw posterior' is subsequently projected on the set of nondecreasing functions on $[0,1]$, piecewise constant on the intervals $I_j$, $1\le j\le J$, via weighted isotonic regression. This projection $f_{\th}^*$  is computed using Lemma 2.1 of \cite{piet_geurt:14}. This boils down to computing the left derivative of the greatest convex minorant of the cusum diagram consisting of the points $P_j$, for $0\le j\le J$ with $P_0=(0,0)$ and
 \begin{align*}
 	P_j=\left(\frac1n\sum_{k=1}^j n_k,\frac1n\sum_{k=1}^j n_k\th_k\right),\,\, 1\le j\le J
 \end{align*}
 if all $n_j>0$, see (3.2) in \cite{moumita:21} (note that Lemma 2.1 in \cite{piet_geurt:14} has the condition that all weights are {\it strictly} positive). It is clear that computing the isotonic regression can be restricted to those $j$ with $n_j>0$ and that for those $j$ with $n_j=0$, $f_{\th}^*$ can be given any value such that monotonicity is not violated.

 	In this procedure, various choices need to be made. One is the number of intervals $J$. In \cite{moumita:21}, the asymptotic bounds 
 	\begin{align}
 		\label{bounds_for_J}
 		n^{1/3}\ll J\ll n^{2/3}
 	\end{align}
 	are given and the closest integer to $n^{1/3}\log n$ is chosen in the simulations. Here the symbol ``$\ll$'' means ``is of lower order than", as $n\to\infty$. Also the noise variance $\sigma_0^2$ needs to be dealt with. For this, \cite{moumita:21}
 	choose the natural empirical Bayes estimate (fixing $\zeta$ and $\Lambda$), given by
 	\begin{align}
 		\label{eq:empBayesSigma2}
 		\hat\s_n^2=n^{-1}(\bm y-\bm B\bm\zeta)^T(\bm B\bm\Lambda\bm B^T+\bm I)^{-1}(\bm y-\bm B\bm\zeta)
 		\qquad \bm\Lambda=\text{diag}(\lambda_1^2,\dots,\lambda_J^2),
 	\end{align}
 	where ${\bm B}=(b_{ij})$ the $n\times J$ `design matrix'
 	with entries $b_{ij}=1_{I_j}(x_i)$, corresponding to the regression model
 	$
 	\bm y=\bm B\bm\th+\bm\epsilon
 	$ following from the representation $f_{\bm\th}(x_i)=({\bm B}{\bm \theta})_i$;
 	see the Appendix. As also shown in the Appendix, this estimate can be rewritten as
 	\begin{align}
 		\label{eq:exprEmpBayesSigma2}
 			\hat\s_n^2&=\frac1n
 			\sum_{j=1}^J\sum_{x_i\in I_j}\left(y_i-\bar{y}_j\right)^2+\frac1n\sum_{j=1}^J\frac{n_j(\bar{y}_j-\zeta_j)^2}{1+n_j\lambda_j^2}\\
 			&= \frac1n \sum_{i=1}^n\left(y_i-\bar{f}(x_i)\right)^2+\frac1n\sum_{j=1}^J\frac{n_j(\bar{y}_j-\zeta_j)^2}{1+n_j\lambda_j^2}\nonumber
 	\end{align}
 where $\bar{f}(x)=\sum_{j=1}^J\bar{y}_j1_{I_j}(x)$ is the aforementioned maximum likelihood estimate of $f_0$ over all piecewise constant (not necessarily monotone) functions on the intervals $I_j$. The first term in this expression is the mean of the squared residuals of the observations with respect to $\bar{f}$. This is a quite natural estimator of the variance. The influence of the hyper parameters $\bm \zeta$ and $\bm \Lambda$ on the estimate of $\sigma_0^2$ can be inferred from the second term in the expression.
 
 As shown in the Appendix, the empirical Bayes estimate for $\zeta$, not taking into account monotonicity  is given by 
 \begin{align}
 	\label{eq:empBayesZeta}
 \bar{\zeta}=(\bar{y}_1,\bar{y}_2,\ldots,\bar{y}_J)^T.
 \end{align}
Substituting this in (\ref{eq:exprEmpBayesSigma2}) makes the second term vanish. Using the empirical Bayes estimate over the monotone vectors $\zeta$, being the isotonic regression of $\bar{\zeta}$ with weights $n_j/(1+n_j\lambda_j^2)$ increases the empirical Bayes estimate for $\sigma_0^2$.

	With the choices  $\bm \zeta=0$ and $\lambda_j\equiv\lambda$ made in \cite{moumita:21}, the empirical Bayes estimate for $\sigma_0^2$ becomes
 \begin{align}
 	\label{eq:empbay}
 		\hat{\sigma}_n^2=\frac1n \sum_{i=1}^n\left(y_i-\bar{f}(x_i)\right)^2+\frac1n\sum_{j=1}^J\frac{n_j\bar{y}_j^2}{1+n_j\lambda^2}
 \end{align}
 	For relatively large values of $\lambda^2n_j$, the second term becomes negligible to the first.

 	Because the density $g$ generating the $X_i$'s is nonvanishing on $[0,1]$, the (random) number  $N_j$ of points in intervals of length of the order $1/J=1/J_n$ is (in the setting of \cite{moumita:21}) of the order $n/J_n$. With the restriction $n^{1/3}<< J_n<< n^{2/3}$, this means that $N_j$ will be of bigger order than $n^{1/3}$; taking $J_n\approx n^{1/3}\log n$, $N_j$ will be of order $n^{2/3}/\log n$. Therefore, for reasonable choice of $\lambda$,  $\lambda^2N_j>>1$ with high probability when $n$ is large. 
	
 	Considering (\ref{posterior_theta}) with $\zeta_j\equiv0$ and $\lambda_i\equiv \lambda$, fixed, as chosen in \cite{moumita:21}, a draw from the raw posterior of $\theta_j$ can be viewed as
 	$$
 	\theta_j=\bar{f}(j/J)+\tilde{\epsilon}_j
 	$$
 	where 
 	\begin{equation}
 		\label{eq:naive-est-equation}
 		\bar{f}(j/J)=\frac{\bar{y}_j}{1+1/(n_j\lambda^2)}\mbox{ and }\tilde{\epsilon}_j\sim^{\rm indep} N\left(0,\frac{\sigma_0^2/n_j}{1+1/(n_j\lambda^2)}\right),
 	\end{equation}
 where $\sigma_0^2$ is estimated by its Empirical Bayes estimate (\ref{eq:empbay}).
 Again due to the restriction $n^{1/3}<< J_n<< n^{2/3}$,  $\bar{f}$ in (\ref{eq:naive-est-equation}) is a (generally non-monotone) local average estimator of $f_0$. The added noise $(\tilde{\epsilon}_j)$ is normal and reflects the variance of the original $\bar{Y}_j$. This means that the draw from the projected posterior is computed as left derivative of the cumulative sum diagram consisting of the points $P_j$, $0\le j\le J$ with $P_0=(0,0)$ and 
 \begin{align}
 	\label{eq:diagramBayes}
 	P_j=\left(\sum_{k=1}^jn_k,\sum_{k=1}^jn_k\left(\frac{\bar{y}_k}{1+1/(n_k\lambda^2)}+\tilde{\epsilon}_k\right)\right)
 \end{align}

 	In \cite{moumita:21}, the following example is considered:
 	\begin{align*}
 		f_0(x)=x^2+x/5,\qquad x\in[0,1].
 	\end{align*}
 	Here the $X_i$ are independently uniformly distributed on $[0,1]$ and the $\e_i$ have a normal $N(0,0.01)$ distribution. The choices for $J$, $\zeta_j$ and $\l_j$ are $J=\lfloor n^{1/3}\log n\rfloor$, $\zeta_j=0$ and $\l_j=10$, following the parametrization in the {\tt R} code, kindly sent to us by Moumita Chakraborty.
 	A picture of a single draw $f_{\bm\th}$ from the raw posterior and its isotonic projection $f_{\bm\th}^*$, for a sample of size $n=500$ is shown in Figure \ref{figure:posterior}.

 	\begin{figure}[!ht]
 		\begin{subfigure}[b]{0.45\textwidth}
 			\includegraphics[width=\textwidth]{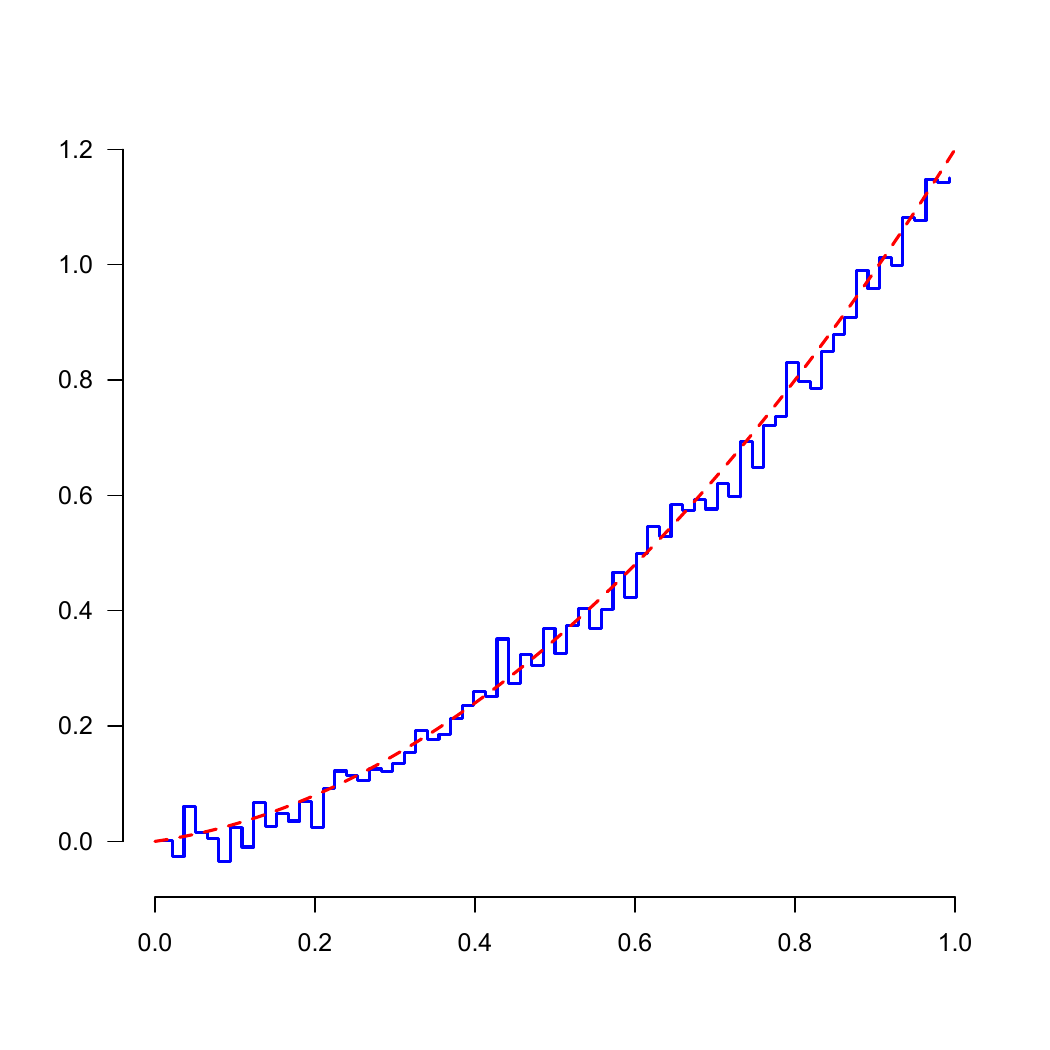}
 			\caption{}
 			\label{fig:posterior_estimate}
 		\end{subfigure}
 		\begin{subfigure}[b]{0.45\textwidth}
 			\includegraphics[width=\textwidth]{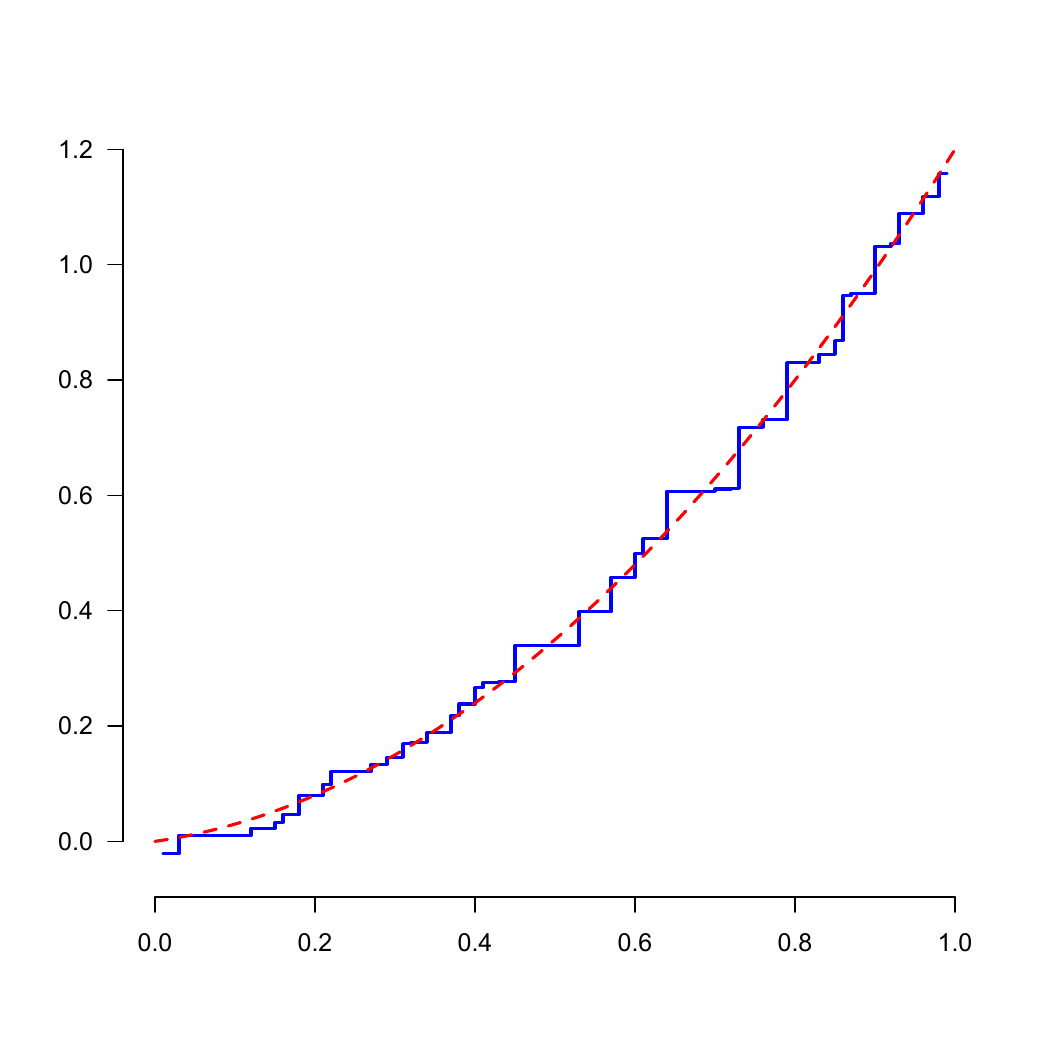}
 			\caption{}
 			\label{fig:isotonic_projection_posterior}
 		\end{subfigure}
 		\caption{(a) A single draw  $f_{\bm\th}$ (blue) from the raw posterior and (b) the  corresponding projection $f_{\bm\th}^*$ (blue), for a sample of size $n=1000$ and $f_0(x)=x^2+x/5$ (red, dashed).}
 		\label{figure:posterior}
 	\end{figure}
 	
 	Now, one can generate $1000$ posterior samples of $\bm\th=(\th_1,\dots,\th_J)$ from the posterior normal distribution, specified in (\ref{posterior_theta}), and consider  the $\tfrac12\a$th and $(1-\tfrac12\a)$th percentiles of the isotonic projections $\hat f_{\bm\th}^*(t)$ at a fixed point $t$. Would this give us, at least asymptotically, valid $95\%$ confidence intervals for $f_0(t)$?
 	
 	The question is answered in \cite{moumita:21} by Theorem 3.3 on p.\ 1017:  $\Pi\left(n^{1/3}\{\hat f_{\bm\th}^*(t)-f_0(t)\}\le z|D_n\right)$ converges to a limit distribution, leading to wider intervals than in the situation in which we have the Chernoff distribution as limit. The fraction by which they become wider is given in \cite{moumita:21}.

\section{Credible intervals as bootstrap percentile confidence intervals}
\label{sec:credboot}
\subsection{The percentile bootstrap for the regression model}
\label{subsec:percentile_regression}
In the Bayes approach, we considered random parameters $\th_j$, with (posterior) distribution given in (\ref{posterior_theta}).
In the simulations, accompanying the paper \cite{moumita:21}, the prior parameter $\bm\zeta$ was taken $\bm\zeta=\bm0$ and $\l_j\equiv\l>0$.
Moreover, the empirical Bayes estimator $\hat\s_n^2$ was taken as estimator for $\sigma_0^2$.
 
With these choices, we get:
\begin{align*}
\th_j\sim N\left(\frac{\bar y_j}{1+1/(n_j\l^2)},\frac{\hat\s_n^2}{n_j\bigl\{1+1/(n_j\l^2)\bigr\}}\right),  1\le j\le J,\,\, {\rm independently}.
\end{align*}
This means asymptotically, in first order:
\begin{align}
\label{Bayes_posteriors}
\th_j\sim N\left(\bar y_j,\frac{\hat\s_n^2}{n_j}\right),\,\,1\le j\le J
\end{align}
if we keep $\l$ bounded away from zero ($\l^2=100$ was taken in the simulations with paper \cite{moumita:21}). Next the confidence intervals were determined by taking the percentiles of simulated values of the (weighted) monotonic projections of the $\th_j$'s with distribution given by (\ref{Bayes_posteriors}).

Algorithmically, this can be viewed as a percentile bootstrap method, where a bootstrap sample is generated by adding noise to an estimate of the regression function. The regression estimate in this setting is the (weighted) least squares estimate of $f_0$, piecewise constant on intervals $I_j$ and not taking into account the monotonicity constraint (so: $\bar{y}_j$ on $I_j$). The noise is sampled from a centered normal distribution with estimated variance. Then, $\hat f^{*}_n$ is determined by computing the (weighted) isotonic regression based on the bootstrap dataset. Adopting the ''bootstrap notation'' rather than the ''Bayesian notation'' $\theta_j$, define
where
\begin{align}
	\label{distribution_theta}
	Y_j^*\sim N(\bar Y_j,\hat\s_n^2/N_j),\qquad j=1,\dots,J
\end{align}
and note that given the original data, $(Y_1^*,\dots,Y_J^*)=^D(\theta_1,\ldots,\theta_J)$, in view of (\ref{Bayes_posteriors}) and (\ref{distribution_theta}). Using this notation, $\hat f_n^*$ is found by taking the left derivative of the convex minorant of the cusum diagram, running through the points
\begin{align}
	\label{cusum:percentile}
	P_j^*=\left(\sum_{i=1}^j N_i,\sum_{i=1}^j N_i Y_i^*\right),\qquad j=1,\dots,J.
\end{align}

To study the asymptotic behavior of $\hat{f}^*_n$, we  define the (local) ``bootstrap'' process
\begin{align}
\label{credible_process}
\widetilde W_n^*(t)=n^{-1/3}\Biggl\{\sum_{j:j/J\in[0,t_0+n^{-1/3}t]}N_j\left\{Y_j^*-\bar Y_j\right\}-\sum_{j:j/J\in[0,t_0]}N_j\left\{Y_j^*-\bar Y_j\right\}\Biggr\}
\end{align}
and the (local) ``sample'' process $\widetilde W_n$  by:
\begin{align}
\label{2nd_martingale}
\widetilde W_n(t)&=n^{-1/3}\Biggl\{\sum_{j:j/J\in[0,t_0+n^{-1/3}t]}1_{\{N_j>0\}}N_j\Bigl\{\bar Y_j-\frac{n}{N_j}\int_{I_j}f_0(u)\,d\G_n(u)\Bigr\}\\
&\qquad\qquad\qquad\qquad\qquad-\sum_{j:j/J\in[0,t_0]}1_{\{N_j>0\}}N_j\Bigl\{\bar Y_j-\frac{n}{N_j}\int_{I_j}f_0(u)\,d\G_n(u)\Bigr\}\Biggr\}.
\end{align}

With these definitions we have the following theorem, similar to Theorem 3.3 in \cite{moumita:21}.

\begin{theorem}
\label{theorem:limit_posterior}
Let $D_n=\{(X_1,Y_1),\dots,(X_n,Y_n)\}$ and let $\hat f_n^*$ be a draw generated according to the bootstrap procedure described above. Then, for each fixed $x\in\R$, as $n\to\infty$,
\begin{align}
\label{scaled_bootrstrap_values}
&\P\left(n^{1/3}\bigl\{\hat f_n^*(t_0)-f_0(t_0)\bigr\}\le x\Bigm|D_n\right)\nonumber\\
&\stackrel{{\mathcal D}}\longrightarrow \P\left[\left(\frac{4\s_0^2f_0'(t_0)}{g(t_0)}\right)^{1/3}\text{\rm argmin}_{t\in\R}\left[W_1(t)+W_2(t)+t^2\right]\le x\Bigm| W_1\right]
\end{align}
where $W_1$ and $W_2$ are independent standard two-sided Brownian motions.
\end{theorem}

Note that, for $t>0$,
\begin{align}
&n^{-1/3}\Biggl\{\sum_{j:j/J\in[0,t_0+n^{-1/3}t]}N_j\bigl\{Y_j^*-f_0(t_0)\bigr\}-\sum_{j:j/J\in[0,t_0]}N_j\bigl\{Y_j^*-f_0(t_0)\bigr\}\Biggr\}\nonumber\\
&=
\widetilde W_n^*(t)+\widetilde W_n(t)+n^{-1/3}\sum_{j:j/J\in[t_0,t_0+n^{-1/3}t]}n\,\int_{I_j}\{f_0(u)-f_0(t_0)\}\,d\G_n(u)\nonumber\\
&\sim \widetilde W_n^*(t)+\widetilde W_n(t)+\tfrac12 f_0'(t_0) g(t_0) t^2,
\end{align}
with a similar expansion for $t\le0$.

So the percentile bootstrap estimates have the same behavior as the Bayes estimates in \cite{moumita:21}. Histograms of estimates of the posterior probabilities for the Bayesian procedure in \cite{moumita:21} and the corresponding conditional probabilities of the percentile bootstrap in Lemma \ref{theorem:limit_posterior}  for varying $D_n$ of size $n=2000$ and $t_0=0.5$ are shown in Figure \ref{figure:posterior_df_Dn}. The estimates are  the relative frequencies in $1000$ posterior, resp.\ percentile bootstrap samples for each of the original (1000) samples.

	\begin{figure}[!ht]
 		\begin{subfigure}[b]{0.45\textwidth}
 			\includegraphics[width=\textwidth]{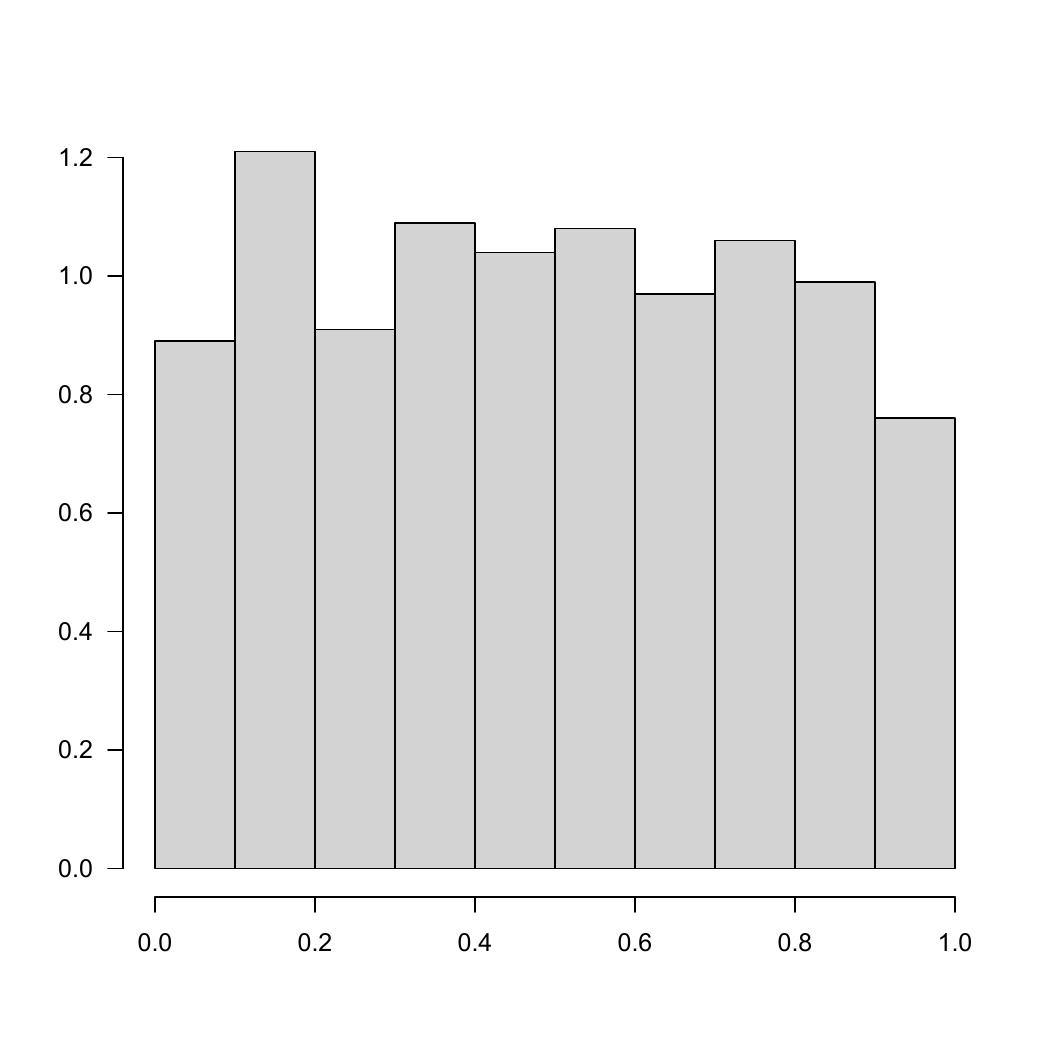}
 			\caption{}
 			\label{fig:credible-real_function_Dn}
 		\end{subfigure}
 		\begin{subfigure}[b]{0.45\textwidth}
 			\includegraphics[width=\textwidth]{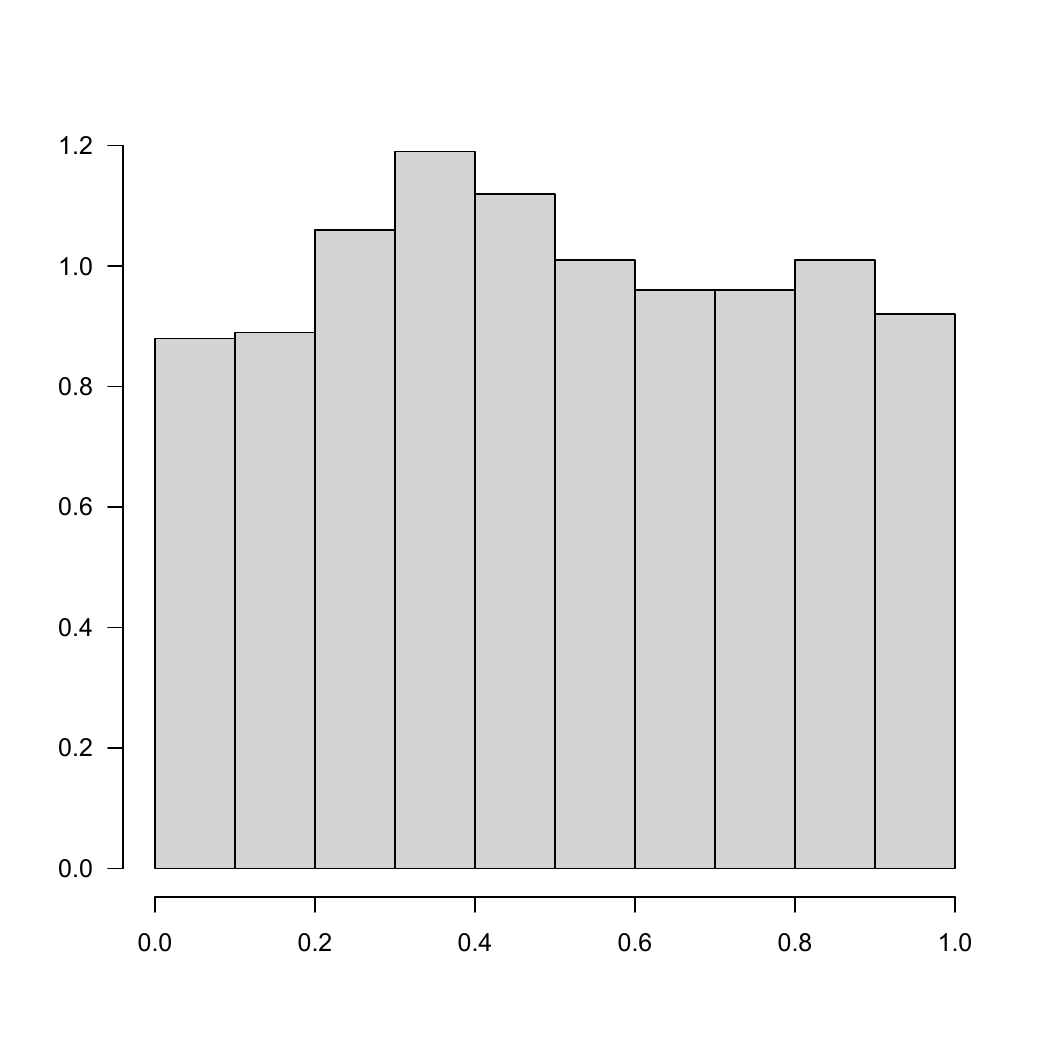}
 			\caption{}
 			\label{fig:credible-LSE_Dn}
 		\end{subfigure}
 		\caption{(a) Histogram of estimate of $\Pi\{n^{1/3}\bigl\{f_n^*(t_0)-f_0(t_0)\bigr\}\le0\bigm|D_n\}$ for the projection-posterior samples in \cite{moumita:21}, (b) Histogram of estimate of $\P\{n^{1/3}\bigl\{\hat f_n^*(t_0)-f_0(t_0)\bigr\}\le0\bigm|D_n\}$ for the percentile bootstrap. Both histograms are based on $1000$ samples $D_n$ of size $n=2000$ and $t_0=0.5$. }
\label{figure:posterior_df_Dn}
\end{figure}

\begin{remark}

{\rm Let $\dd_n=\Pi\{n^{1/3}\{f_n^*(t_0)-\hat f_n(t_0)\}\le0|D_n\}$. In Figure 1 on p.\ 1017 of \cite{moumita:21} three pictures of $\dd_n$ are shown for three different sets of simulated data, where $\hat f_n$ is the LSE. Is is not completely clear to us how $\dd_n$ is sampled here. Since we do not have an explict expression for $\dd_n$, it seems that  an estimate of $\dd_n$ has to be based on a sample of posterior draws $f^*(t_0)$. If we use such a procedure and consider the fluctuation of $\dd_n$ as a function of $D_n$, we get a histogram similar to the histograms in Figure 1 of \cite{moumita:21}. The estimates are relative frequencies in $1000$ samples of size $2000$. See Figure \ref{figure:posterior_df}.
}
\end{remark}

\begin{figure}[!ht]
\includegraphics[width=0.5\textwidth]{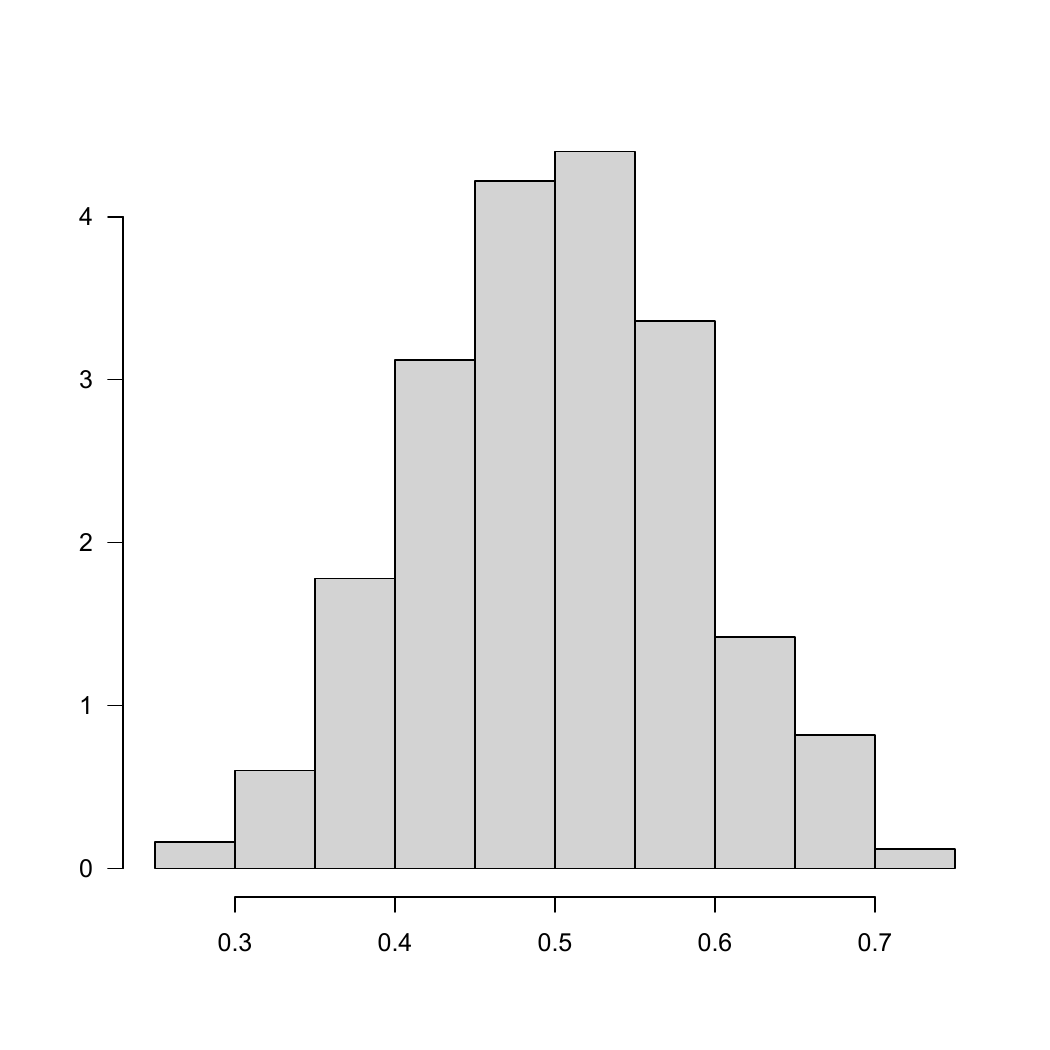}
\caption{Histogram of estimate of $\Pi\{n^{1/3}\bigl\{f_n^*(t_0)-\hat f_n(t_0)\bigr\}\le0\bigm|D_n\}$, for varying $D_n$ of size $n=2000$ and $t_0=0.5$. The estimates are based on relative frequencies in $1000$ draws.}
\label{figure:posterior_df}
\end{figure}

Theorem \ref{theorem:limit_posterior} is the consequence of the following two lemmas.

\begin{lemma}
\label{lemma:localBM_convergence_theta}
Let $W$ be standard two-sided Brownian motion on $\R$, originating from zero. Let  $D(\R)$  be  the space of right continuous functions, with left limits (cadlag functions) on $\R$, equipped with the metric of uniform convergence on compact sets, and let $t_0\in(0,1)$.  Let $\widetilde W_n^*$ be defined by (\ref{credible_process}).
Then, along almost all sequences $(X_1,Y_1),(X_2,Y_2),\dots$, the process $\widetilde W_n^*$ defined by (\ref{credible_process}) converges in $D(\R)$ in distribution conditionally to the process $V$, defined by
\begin{equation}
\label{V_n_limit}
V(t)=\s_0\sqrt{g(t_0)}\,W(t),\,t\in\R.
\end{equation}
Here $W$ is standard two-sided Brownian motion, originating from zero.
\end{lemma}

\begin{proof}
 We consider the case $t\ge0$. It is clear that $t\mapsto \widetilde W_n(t)$ is a martingale with respect to the family of $\s$-algebras $ {\cal F}^*_{n,t},\,t\ge0$, defined by:
 \begin{align*}
 {\cal F}^*_{n,t}=\s\left\{(j/J,\bar y_j):j/J\in(t_0,t_0+n^{-1/3}t]\right\}.\qquad t\ge0.
 \end{align*}
 The quadratic variation process is, for $t\ge0$  given by:
\begin{align*}
 &\left[\widetilde W_n^*\right](t)=n^{-2/3}\sum_{j:j/J\in(t_0,t_0+n^{-1/3}t]}N_j^2\left\{\th_j^*-\bar y_j\right\}^2
 \sim n^{-2/3}\sum_{j:j/J\in(t_0,t_0+n^{-1/3}t]}N_j\s_0^2\,.
\end{align*}
If, for example as in (\cite{moumita:21}), $J\sim n^{1/3}\log n$, we get:
\begin{align*}
N_j\sim n^{2/3}g(t_0)/\log n,
\end{align*}
and
\begin{align*}
n^{-2/3}\sum_{j:j/J\in(t_0,t_0+n^{-1/3}t]}N_j\s_0^2\sim (\log n)^{-1} \sum_{j:j/J\in(t_0,t_0+n^{-1/3}t]}g(t_0)\s_0^2\sim \s_0^2 g(t_0)\,t.
\end{align*}
The case $t<0$ is treated similarly. The result now follows from Rebolledo's theorem, see Theorem 3.6, p.\ 68 of \cite{piet_geurt:14} and \cite{rebolledo:80}. 
 \end{proof}

\begin{lemma}
\label{lemma:localBM_convergence_y_bar}
Let $W$, $t_0$ and $V$ be as defined in Lemma \ref{lemma:localBM_convergence_theta} and $\widetilde W_n$  by (\ref{2nd_martingale}).
Then the process $\widetilde W_n$ converges in $D(\R)$ in distribution, conditionally on the sequence $X_1,X_2,\dots$, to the process $V$.
\end{lemma}

\begin{proof}[Proof of Lemma \ref{lemma:localBM_convergence_y_bar}]
This is proved in the same way as Lemma \ref{lemma:localBM_convergence_theta}, using 
that (\ref{2nd_martingale}) is a martingale.
\end{proof}

\begin{proof}[Proof of Theorem \ref{theorem:limit_posterior}]
We use the ``switch relation'' (see, e.g., Section 3.8 of \cite{piet_geurt:14} and section 5.1 of \cite{GrWe:92}; the terminology is due to Iain Johnstone to denote a construction introduced in a course given by the first author in Stanford, 1990). The bootstrap estimate $f_n^*$ is computed as left derivative of the greatest convex minorant of cumulative sum diagram (\ref{cusum:percentile}). Let the processes $G_n$ and $V_n^*$ be defined by
\begin{align*}
G_n(t)=\sum_{j/J\le t}N_j/n,\qquad V_n^*(t)=\sum_{j/J\le t}N_jY_j^*/n,\qquad t\in[0,1].
\end{align*}
Moreover, let $U_n^*$ be defined by
\begin{align*}
U_n^*(a)=\text{argmin}\{t\in[0,1]:V_n^*(t)-a G_n(t)\},
\end{align*}
for $a$ in the range of $f_0$.
Then we have the ``switch relation'':
\begin{align*}
\hat f_n^*(t)\ge a \iff G_n(t)\ge G_n(U_n^*(a)) \iff t\ge U_n^*(a),
\end{align*}
(compare with (3.35), p.\ 69 of \cite{piet_geurt:14}). So we get if $a=f_0(t_0)$,
\begin{align*}
&\P\left\{n^{1/3}\{\hat f_n^*(t_0)-f_0(t_0)\}\ge x|D_n\right\}=
\P\left\{\hat f_n^*(t_0)\ge a+n^{1/3}x|D_n\right\}\\
&=\P\left\{U_n^*(a+n^{-1/3}x)\le t_0|D_n\right\}=\P\left\{n^{1/3}\bigl\{U_n^*(a+n^{-1/3}x)-t_0\bigr\}\le0|D_n\right\}\nonumber\\
&=\P\left\{\text{argmin}\left[t\in[0,1]:V_n^*(t)-(a+n^{-1/3}x)\,G_n(t)\right]\le 0|D_n\right\}\\
&=\P\left\{\text{argmin}\left[t\in[0,1]:V_n^*(t)-V_n^*(t_0)-(a+n^{-1/3}x)\{G_n(t)-G_n(t_0)\}\right]\le 0|D_n\right\},
\end{align*}
where the last equality holds since the values of the argmin function do not change if we add constants to the function for which we determine the argmin.

By Lemmas \ref{lemma:localBM_convergence_theta} and \ref{lemma:localBM_convergence_y_bar} we get the local expansion:
\begin{align*}
&\P\left\{n^{1/3}\bigl\{\hat f_n^*(t_0)-f_0(t_0)\bigr\}\ge x\Bigm|D_n\right\}\\
&\sim\P\left[\text{argmin}_t\left[\tilde W_n^*(t)+\tilde W_n(t)+\tfrac12 f_0'(t_0) g(t_0)t^2-x g(t_0)\right]\le 0\bigm|D_n\right],
\end{align*}
which (using Brownian scaling) converges in distribution to
\begin{align*}
&\P\left[\left(\frac{4\s_0^2f_0'(t_0)}{g(t_0)}\right)^{1/3}\text{\rm argmin}_{t\in\R}\left[W_1(t)+W_2(t)+t^2\right]\le -x\Bigm| W_1\right]\\
&=\P\left[\left(\frac{4\s_0^2f_0'(t_0)}{g(t_0)}\right)^{1/3}\text{\rm argmin}_{t\in\R}\left[W_1(t)+W_2(t)+t^2\right]\ge x\Bigm| W_1\right].
\end{align*}
The last line uses the type of symmetry used in the proof of Theorem 5.2 of \cite{GrWe:92}.
\end{proof}

\vspace{0.5cm}
In the proof of Theorem \ref{theorem:limit_posterior} we use the tightness of $n^{1/3}\{U_n^*(a+n^{-1/3}x)-t_0\}$, which can be proved along entirely similar lines as the proof of Lemma 3.5 in \cite{piet_geurt:14}.

\subsection{Convergence of a classical percentile bootstrap}
\label{subsec:classic_bootstrap}
It is of interest to investigate what happens if we perform a classical empirical bootstrap, where we resample with replacement from the pairs $(X_i,Y_i)$. This situation, where we also treat the $X_i$ as random from the start instead of keeping them fixed, is called the ``correlation model'' in \cite{hall:92book}. In this case we compute the local means
\begin{align}
\label{theta_correlation}
\bar Y^*_j=(N_j^*)^{-1}\sum_{X_i^*\in I_j}Y_i^*,\qquad I_j=((j-1)/J,j/J],\qquad N_j^*=\#\{i:X_i^*\in I_j\},
\end{align}
where the $(X_i^*,Y_i^*)$ are (discretely) uniformly (re-)sampled with replacement from the set $D_n=\{(X_1,Y_i),\dots,(X_n,Y_n)\}$. If $N_j^*=0$ we define $\bar Y_j^*=0$, these values play no role in the isotonization step.

Note that we can write alternatively, if $N_j^*>0$:
\begin{align}
\label{theta_correlation2}
\bar Y_j^*=(N_j^*)^{-1}\sum_{X_i\in I_j}M_{in}^*Y_i,\qquad N_j^*=\sum_{i:X_i\in I_j}M_{in}^*,
\end{align}
where
\begin{align*}
\left(M_{1n}^*,\dots,M_{nn}^*\right)\sim\text{Multinomial}\left(n;n^{-1},\dots,n^{-1}\right).
\end{align*}
This means that
\begin{align*}
\E\left\{N_j^*\bar Y_j^*\bigm|(X_1,Y_1),\dots,(X_1,Y_n)\right\}=N_j\bar Y_j,\qquad j=1,\dots,J.
\end{align*}

The points of the cusum diagram needed to compute the bootstrap realization of the LSE are given by
\begin{align*}
	P_j^*=\left(\sum_{i=1}^j N_i^*,\sum_{i=1}^j N_i^*\bar Y_i^*\right),\qquad j=1,\dots,J.
\end{align*}

In order to study the local asymptotics of the greatest convex minorant of this diagram, we consider the process
\begin{align}
\label{classic_W_n^*}
\widetilde W_n^*(t)&=n^{-1/3}\Biggl\{\sum_{j:j/J\in[0,t_0+n^{-1/3}t]}\sum_{X_i\in I_j}(M_{in}^*-1)(Y_i-a_0)\nonumber\\
&\qquad\qquad\qquad\qquad\qquad\qquad-\sum_{j:j/J\in[0,t_0]}\sum_{X_i\in I_j}(M_{in}^*-1)(Y_i-a_0)\Biggr\},
\end{align}
where $a_0=f_0(t_0)$, and
\begin{align*}
\widetilde W_n(t)&=n^{-1/3}\Biggl\{\sum_{j:j/J\in[0,t_0+n^{-1/3}t]}\sum_{X_i\in I_j}(Y_i-a_0)
-\sum_{j:j/J\in[0,t_0]}\sum_{X_i\in I_j}(Y_i-a_0)\Biggr\}.
\end{align*}

Defining 
\begin{align*}
U_n(a)=\text{argmin}\Bigl\{t\in[0,1]:n^{-1}\sum_{j:j/J\le t} \bigl\{N_j^*\bar Y_j^*-a\,N_j^*\bigr\}\Bigr\},
\end{align*}
results analogous to Lemmas \ref{lemma:localBM_convergence_theta} and \ref{lemma:localBM_convergence_y_bar} hold. For example we get Lemma \ref{lemma:localBM_convergence_classic} (the analogue to Lemma \ref{lemma:localBM_convergence_theta}) which is proved in the Appendix.

\begin{lemma}
\label{lemma:localBM_convergence_classic}
Let $W$, $t_0$ and $V$ be as defined in Lemma \ref{lemma:localBM_convergence_theta} and $\widetilde W_n^*$ be defined by (\ref{classic_W_n^*}).
Then, along almost all sequences $(X_1,Y_1),(X_2,Y_2),\dots$, the process $\widetilde W_n^*$ converges in $D(\R)$ in distribution conditionally to the process $V$
\end{lemma}

So we get the same behavior as in subsection \ref{subsec:percentile_regression}, but the present approach has the interesting feature that we do not have to estimate the variance of the errors separately. We can just resample with replacement from the original sample $(X_1,Y_1),\dots,(X_n,Y_n)$ and compute the estimator $\hat f_n^*$ in the bootstrap samples.

The simulations, based on the regression function $f(x)=x^2+x/5$ with normal noise with expectation $0$ and variance $0.01$ show almost no difference beween the three methods if $n=20,000$, see Figure \ref{figure:CI_percentages}. At smaller sample size, like, e.g., $n=1000$, the overcoverage is still not reached, as can be seen in Figure \ref{figure:CI_percentages2}. So the phenomenon of overcoverage also here only seems to occur with very large sample sizes.
	
\begin{figure}[!ht]
\begin{subfigure}[b]{0.3\textwidth}
\includegraphics[width=\textwidth]{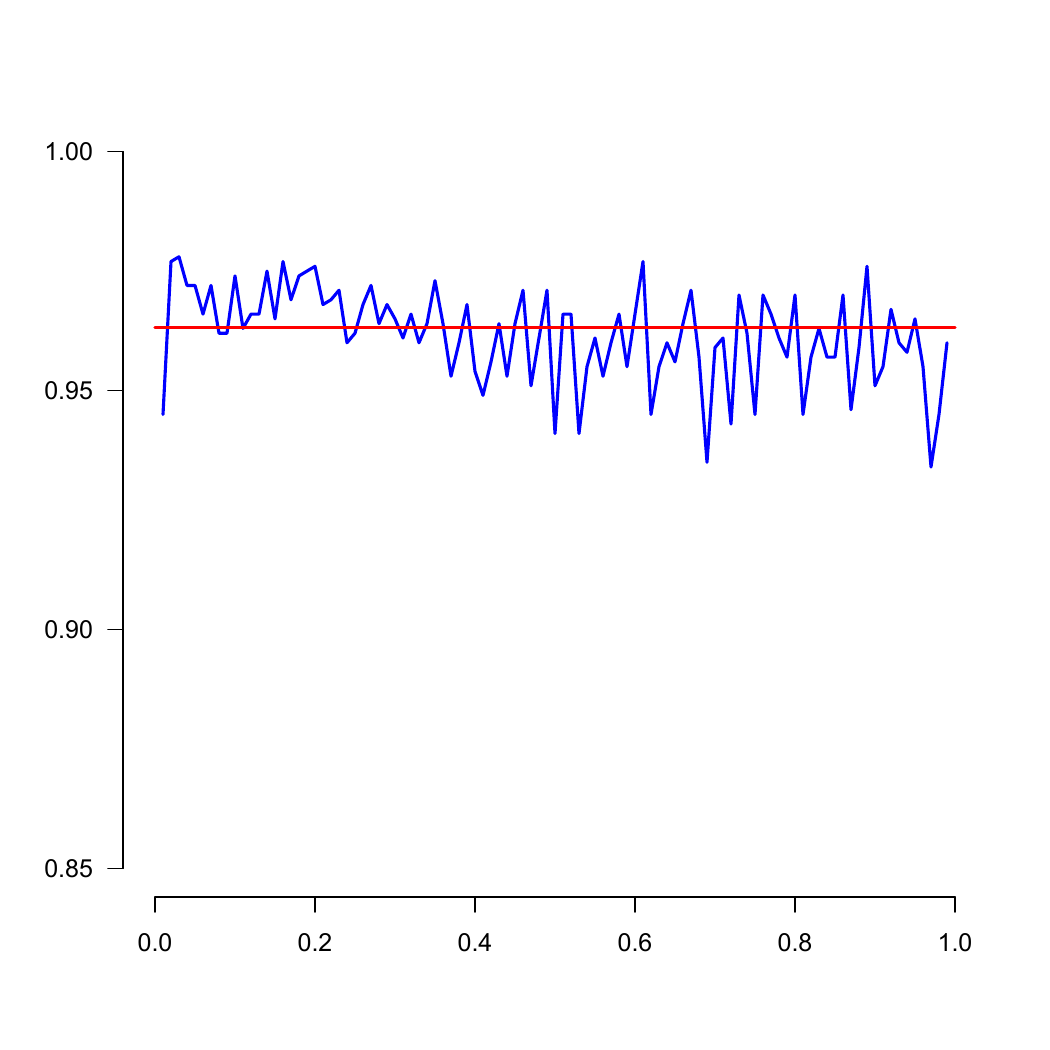}
\caption{}
\label{fig:percentages_credible20,000}
\end{subfigure}
\begin{subfigure}[b]{0.3\textwidth}
\includegraphics[width=\textwidth]{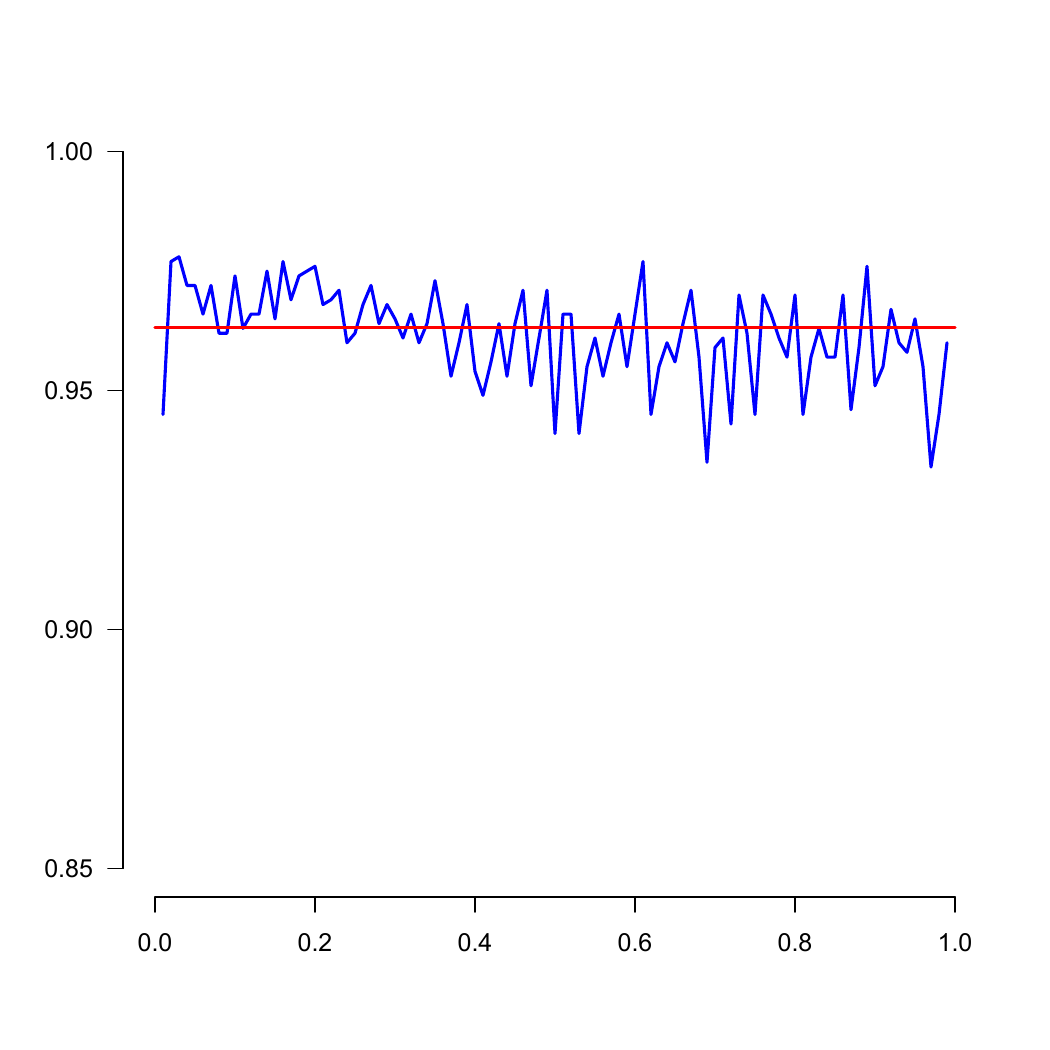}
\caption{}
\label{fig:percentages_percentile_regression20,000}
\end{subfigure}
\begin{subfigure}[b]{0.3\textwidth}
\includegraphics[width=\textwidth]{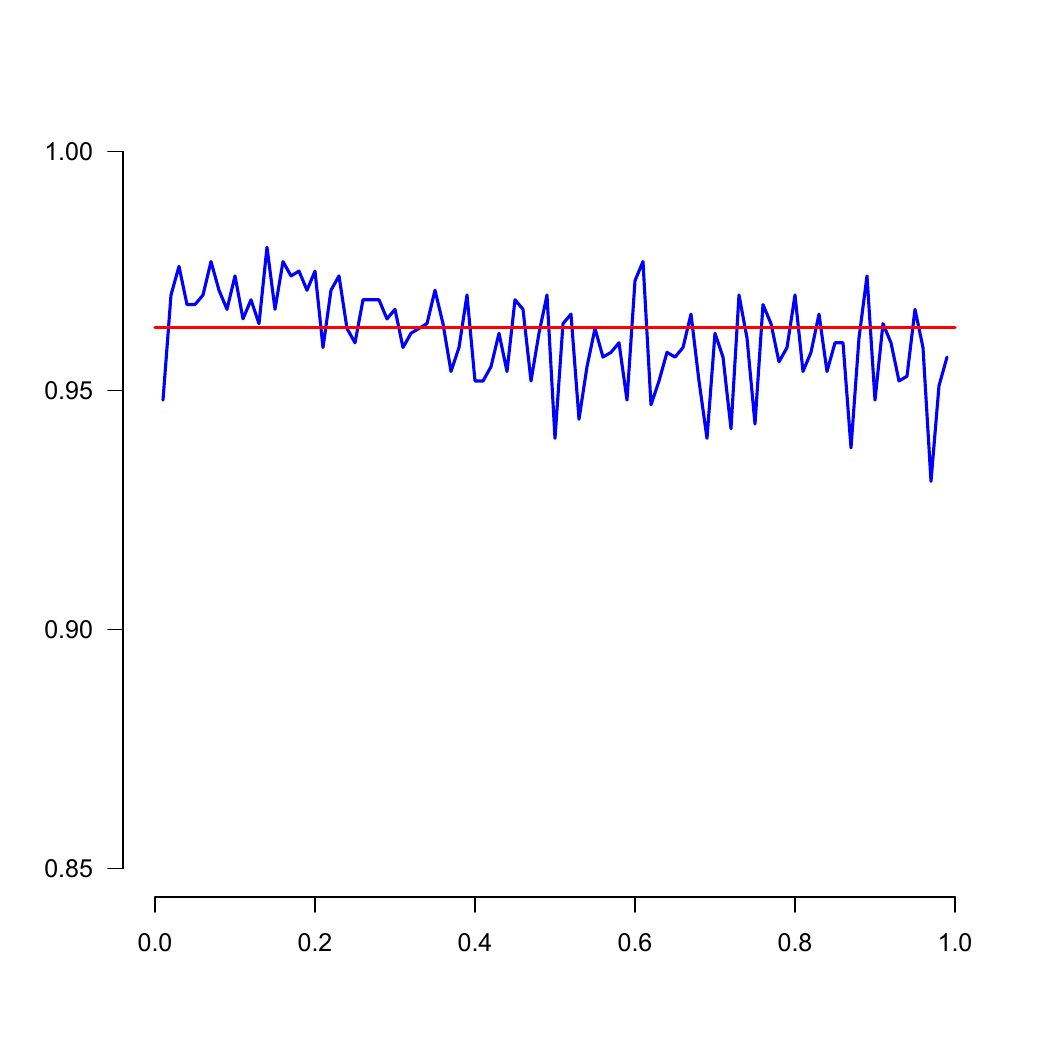}
\caption{}
\label{fig:percentages_percentile20,000}
\end{subfigure}
\caption{Coverage percentages for credible intervals and percentile confidence intervals, for $n=20,000$.  (a) credible intervals, (b) percentile confidence intervals of section \ref{subsec:percentile_regression}, (c) percentile confidence intervals of section \ref{subsec:classic_bootstrap}. The red line is at level 0.96324, but the intervals are based on the $0.025$ and $0.975$ quantiles of the credible, respectively percentile bootstrap simulations. The level 0.96324 was determined from the values of the function $A$ in \cite{moumita:21}.}
\label{figure:CI_percentages}
\end{figure}

\begin{figure}[!ht]
\begin{subfigure}[b]{0.3\textwidth}
\includegraphics[width=\textwidth]{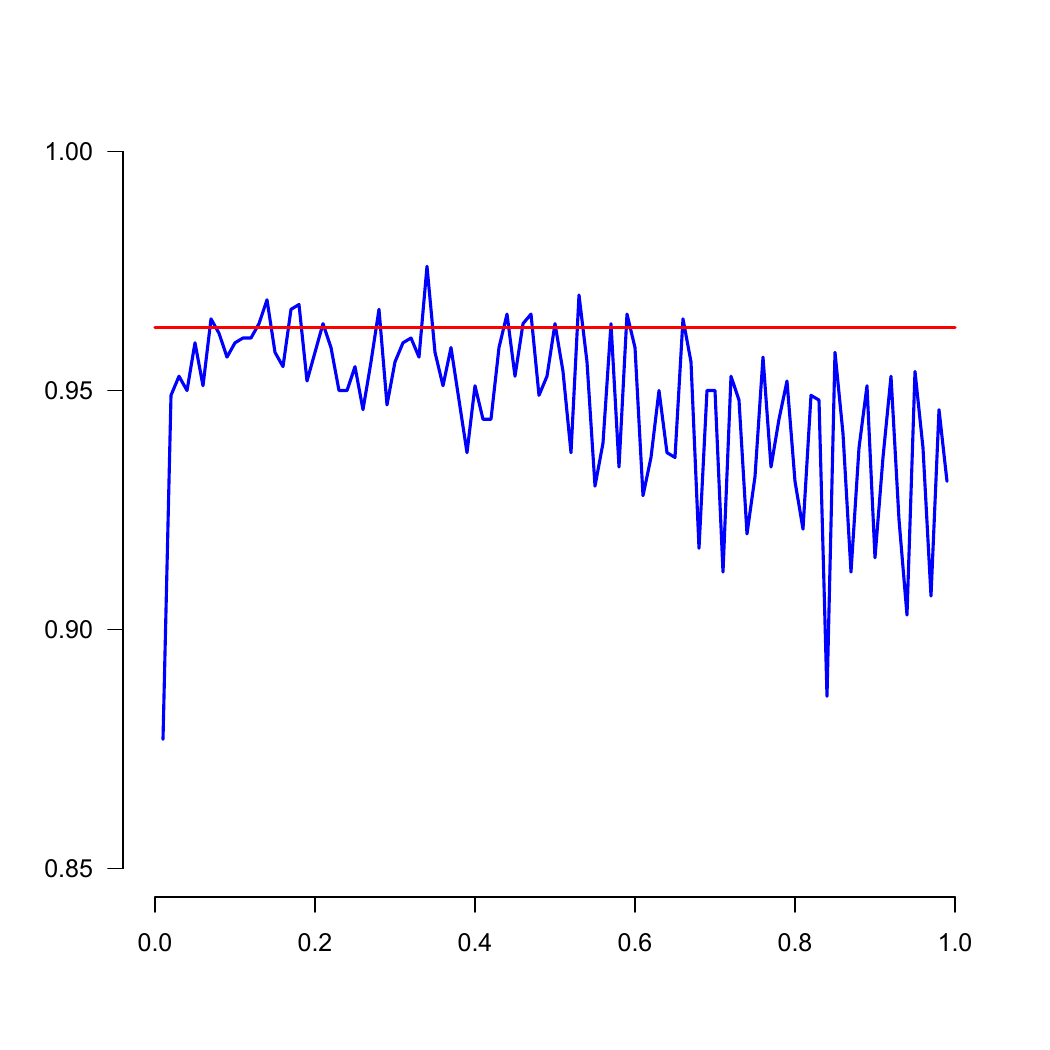}
\caption{}
\label{fig:percentage_credible1000}
\end{subfigure}
\begin{subfigure}[b]{0.3\textwidth}
\includegraphics[width=\textwidth]{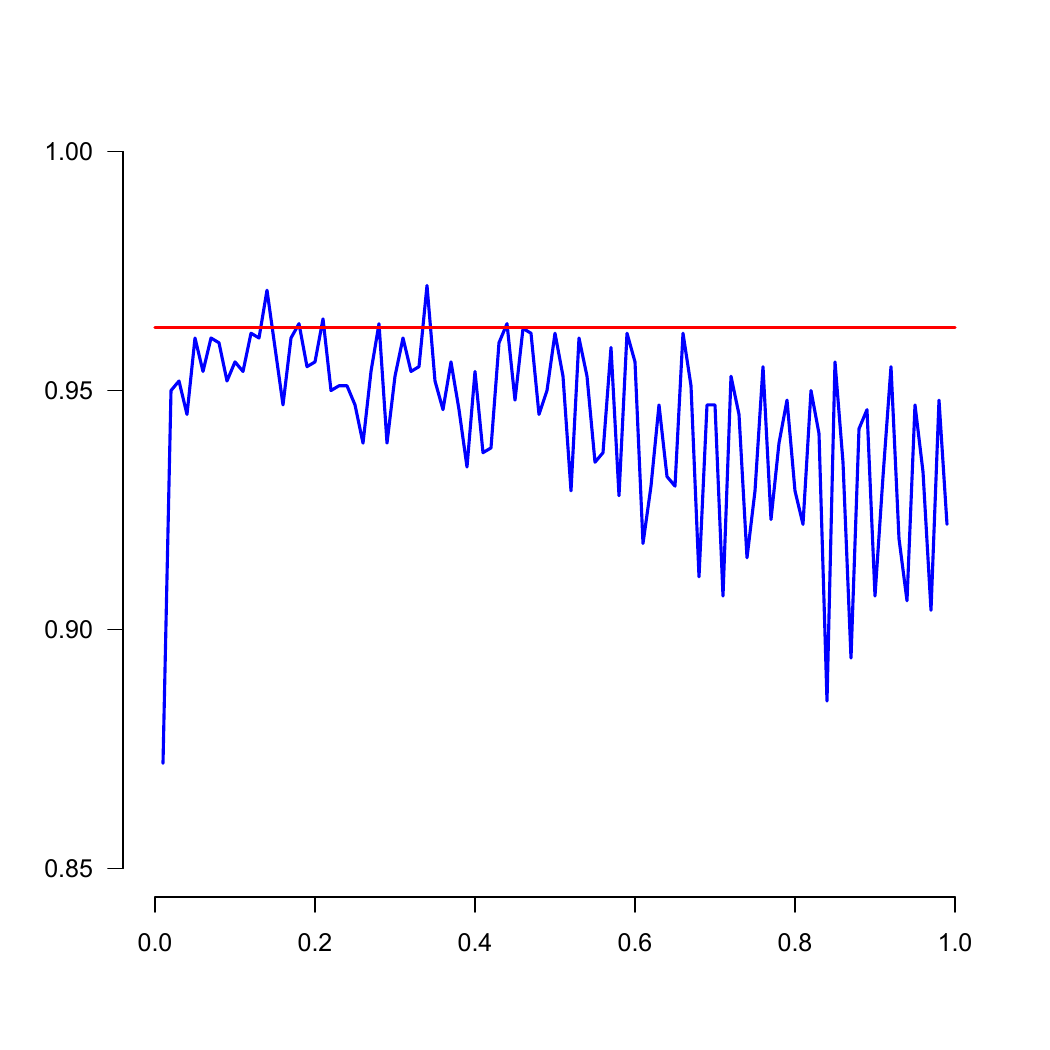}
\caption{}
\label{fig:percentage_percentile_regression1000}
\end{subfigure}
\begin{subfigure}[b]{0.3\textwidth}
\includegraphics[width=\textwidth]{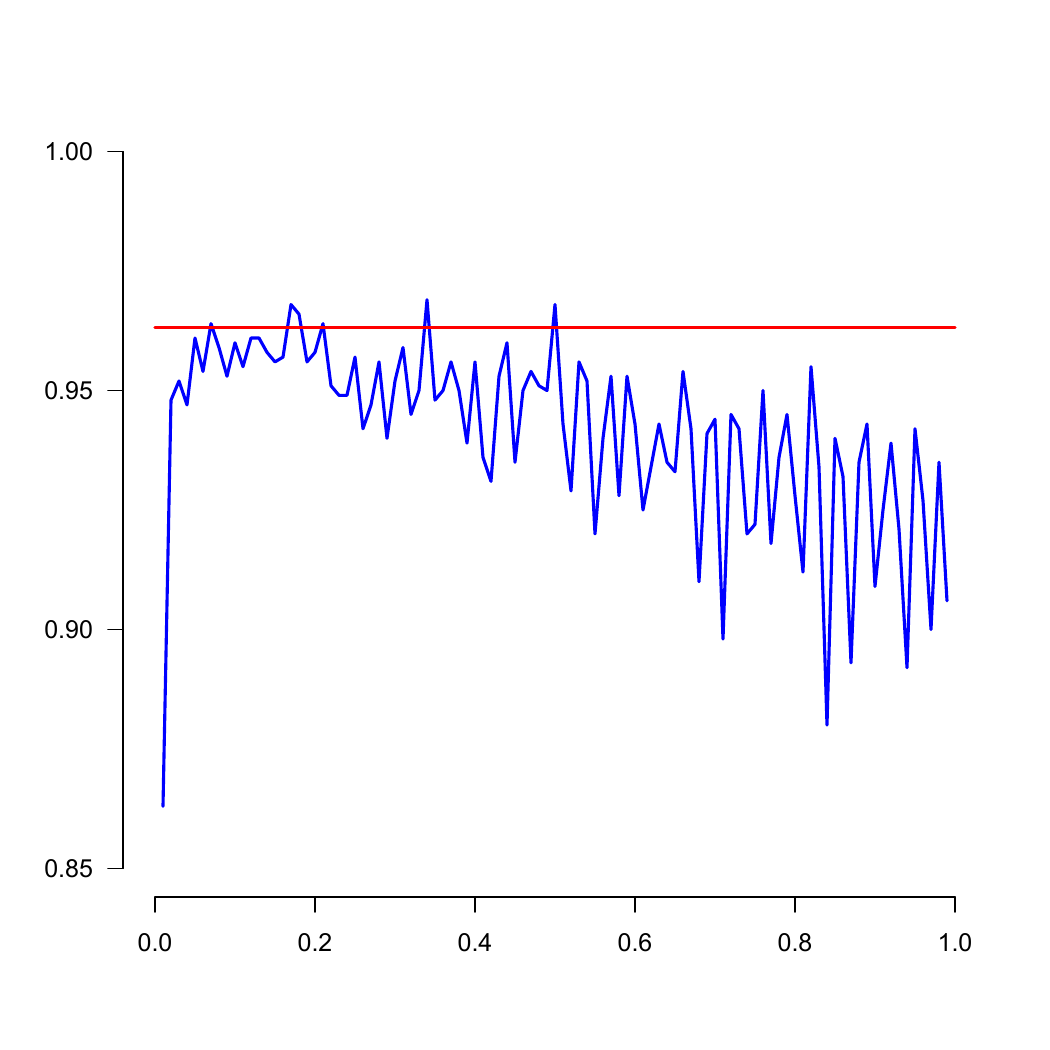}
\caption{}
\label{fig:percentage__percentile1000}
\end{subfigure}
\caption{Coverage percentages for credible intervals and percentile confidence intervals, for $n=1000$.  (a) credible intervals, (b) percentile confidence intervals of section \ref{subsec:percentile_regression}, (c) percentile confidence intervals of section \ref{subsec:classic_bootstrap}. The red line is at level 0.96324, but the intervals are based on the $0.025$ and $0.975$ quantiles of the credible, respectively percentile bootstrap simulations. The level 0.96324 was determined from the values of the function $A$ in \cite{moumita:21}.}
\label{figure:CI_percentages2}
\end{figure}

\section{Cube root $n$ consistent smoothed bootstrap confidence intervals}
\label{sec:CI_cuberoot_n}
In \cite{SenXu2015} it was shown for interval censoring models that cube root $n$ convergent bootstrap confidence intervals can be computed for the distribution function at a fixed point with the right asymptotic coverage. Key in this, is the convergence of the nonparametric maximum likelihood estimator to Chernoff's distribution. We show that a similar approach is possible in the present context.

In the regression context this means that we use, as in \cite{geurt_piet:23}, the smoothed least squares estimator (the SLSE) $\tilde f_{nh}$, for a bandwidth $h=cn^{-1/5}$. 
To define $\tilde f_{nh}$, let $K$ be  a symmetric twice continuously differentiable nonnegative kernel with support $[-1,1]$ such that $\int K(u)\,du=1$. Let $h>0$ be a bandwidth and define the scaled kernel $K_h$  by
\begin{align}
	\label{def_K_h}
	K_h(u)=\frac1h K\left(\frac{u}{h}\right), \,\,\,u\in\R.
\end{align}
The SLSE $\tilde f_{nh}$ is then for $t\in[h,1-h]$ defined by
\begin{align}
\label{def_SLSE}
\tilde f_{nh}(t)=\int K_h(t-x)\,\hat f_n(x)\,dx.
\end{align}
For $t\notin[h,1-h]$ we use the boundary correction, defined in \cite{geurt_piet:23} (see (2.6) and (2.7) in \cite{geurt_piet:23}).

We now generate residuals $E_i$ with respect to (\ref{def_SLSE}), defined by
\begin{align*}
E_i=Y_i-\tilde f_{nh}(X_i),\qquad i=1,\dots,n,
\end{align*}
and compute the centered residuals $\tilde E_i$, 
\begin{align}
\label{residuals}
\tilde E_i=E_i-n^{-1}\sum_{j=1}^n E_j,\qquad i=1,\dots,n.
\end{align}

From these residuals, we generate bootstrap samples
\begin{align}
\label{bootstrap_sample}
(X_i,Y_i^*),\qquad Y_i^*=\tilde f_{nh}(X_i)+\tilde E_i^*,\qquad i=1,\dots,n,
\end{align}
where the $\tilde E_i^*$ are drawn uniformly with replacement from the residuals $\tilde E_i$ defined by (\ref{residuals}). For the bootstrap samples (\ref{bootstrap_sample}), we compute the monotone (non-smoothed) LSE $\hat f_n^*$  and consider the differences
\begin{align}
\label{SLSE_intervals0}
\hat f_n^*(t)-\tilde f_{nh}(t),
\end{align}
and the $95\%$ bootstrap confidence intervals, given by
\begin{align}
\label{first_conf_intervals}
	\left(\hat f_n(t)-Q_{0.975}^*,\hat f_n(t)-Q_{0.025}^*\right),
\end{align}
where $Q_{0.025}^*$ and $Q_{0.975}^*$ are the $2.5$th and $97.5$th percentiles of $1000$ bootstrap samples of (\ref{SLSE_intervals0}) and $\hat f_n$ is the LSE in the original sample. Note that this is the more conventional bootstrap approach, rather than the percentile method.

\begin{figure}[!ht]
\begin{subfigure}[b]{0.3\textwidth}
\includegraphics[width=\textwidth]{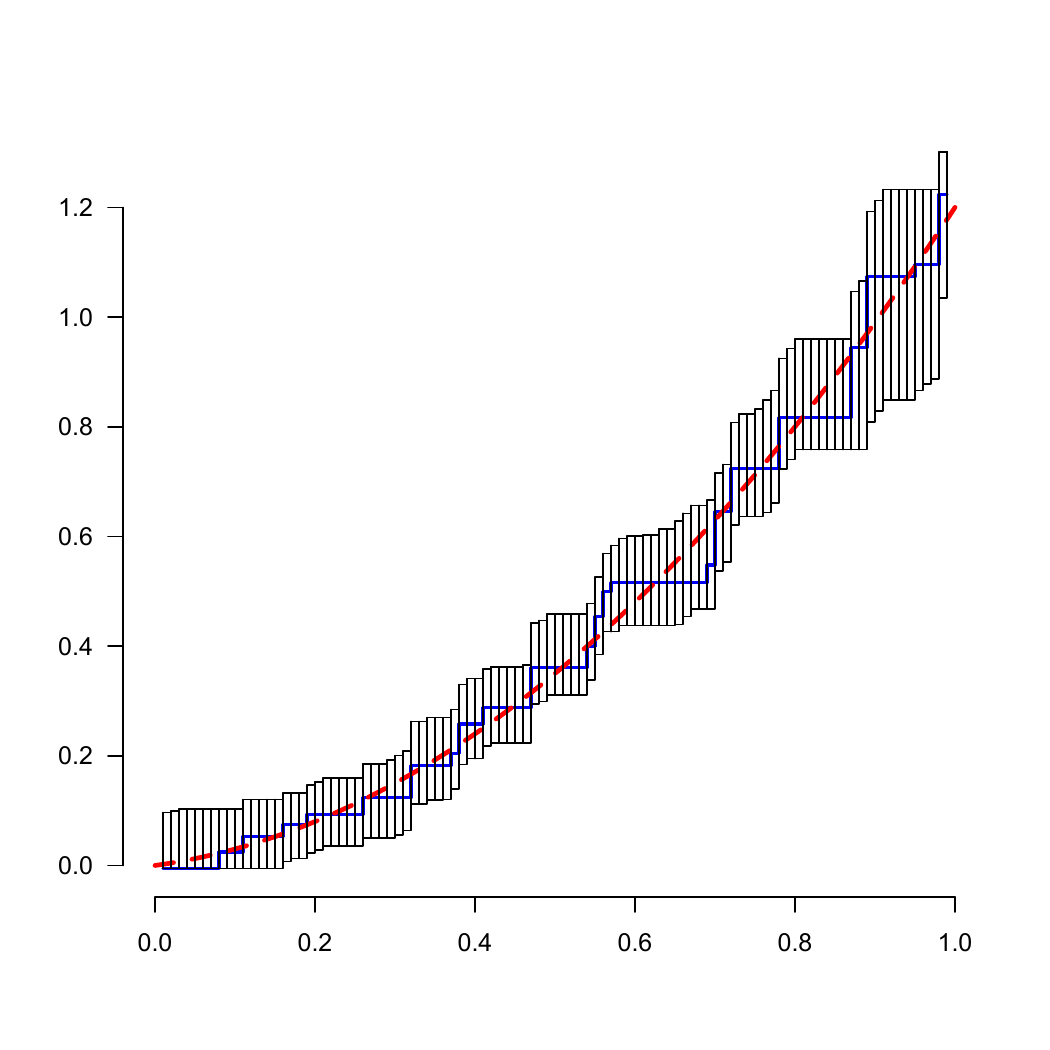}
\caption{}
\label{fig:p1}
\end{subfigure}
\begin{subfigure}[b]{0.3\textwidth}
\includegraphics[width=\textwidth]{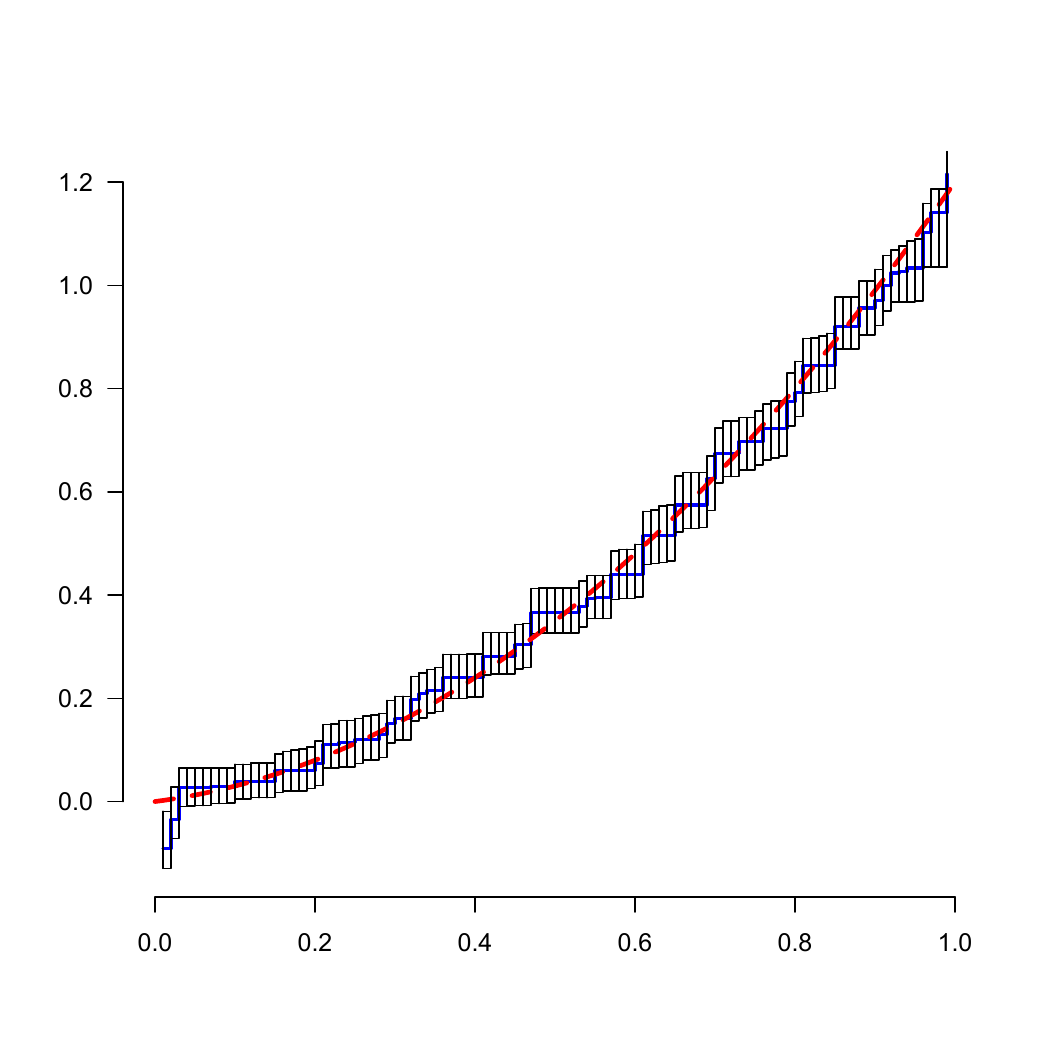}
\caption{}
\label{fig:q1}
\end{subfigure}
\begin{subfigure}[b]{0.3\textwidth}
\includegraphics[width=\textwidth]{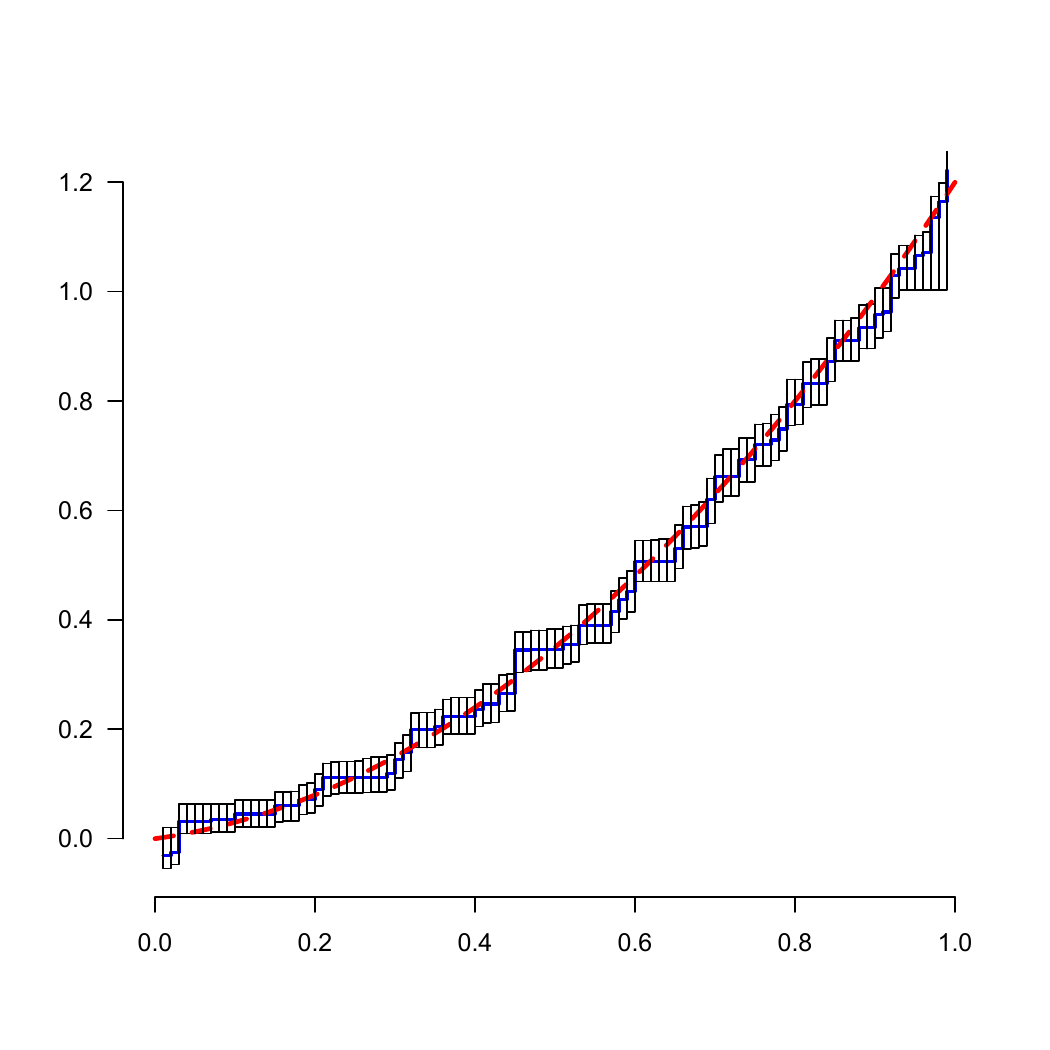}
\caption{}
\label{fig:r1}
\end{subfigure}
\caption{Confidence intervals, based on (\ref{first_conf_intervals}), for $n=100, 500$ and $1000$. The blue curve is the nonparametric monotone LSE, and the red dashed curve the real regression function $f_0(x)=x^2+x/5$. (a) $n=100$, (b) $n=500$, (c) $n=1000$.}
\label{figure:CI_Chernoff}
\end{figure}

The SLSE $\tilde f_{nh}$. which plays a central role in the construction of these confidence intervals, has the limit behavior specified in the following theorem, which is Theorem 1 in \cite{geurt_piet:23}. 
 
 \begin{theorem}
\label{th:limit_SLSE}
	Let $f_0$ be a nondecreasing continuous function on $[0,1]$. Let $X_1,X_2,\ldots$ be i.i.d random variables with continuous density $g$, staying away from zero on $[0,1]$, where the derivative $g'$ is continuous and bounded on $(0,1)$.
	Furthermore, let $\e_1,\e_2\ldots$ be i.i.d. random variables distributed according to a sub-Gaussian distribution with expectation zero and variance $0<\s_0^2<\infty$, independent of the $X_i$'s.	
	Then consider $Y_i$, defined by
	\begin{align*}
		Y_i=f_0(X_i)+\e_i, \,\,\,\, i=1,2,\ldots
	\end{align*}
	Suppose $t_0\in(0,1)$ such that $f_0$ has a strictly positive derivative and a continuous second derivative $f_0''(t_0)\ne0$ at $t_0$. Then, for the SLSE $\tilde f_{nh}$ defined by (\ref{def_SLSE}) based on the pairs $(X_1,Y_1),\ldots,(X_n,Y_n)$, and $h\sim cn^{-1/5}$ for $c>0$,
	\begin{align*}
		n^{2/5}\left\{\tilde f_{nh}(t_0)-f_0(t_0)\right\}\stackrel{{\cal D}}\longrightarrow N(\b,\s^2).
	\end{align*}
	Here
	\begin{align}
		\label{asymp_bias_var}
		\b=\tfrac12c^2 f_0''(t_0)\int u^2K(u)\,du\,\,\, \mbox{ and }\,\,\,\s^2=\frac{\s^2_{0}}{c g(t)}\int K(u)^2\,du.
	\end{align}
The asymptotically Mean Squared Error optimal constant $c$ is given by:
\begin{align*}
c=\left\{\frac{\s^2_{0}}{g(t_0)}\int K(u)^2\,du\Bigm/\left\{f_0''(t_0)\int u^2K(u)\,du\right\}^2\right\}^{1/5}.
\end{align*}
\end{theorem}

We have the following lemma of which the proof is given in the Appendix.

\begin{lemma}
\label{lemma:localBM_convergence}
Let $W$, $t_0$ and $V$ be as defined in Lemma \ref{lemma:localBM_convergence_theta} and $\widetilde W_n^*$ be defined by:
\begin{align*}
\widetilde W_n^*(t)=n^{-1/3}\Biggl\{\sum_{i:X_i\in[0,t_0+n^{-1/3}t]}\tilde E_i^*-\sum_{i:X_i\in[0,t_0]}\tilde E_i^*\Biggr\},\qquad t\in\R,
\end{align*}
where $\tilde E_i^*$ is drawn with replacement from the residuals $\tilde E_i$, defined by (\ref{residuals}).
Then, along almost all sequences $(X_1,Y_1),(X_2,Y_2),\dots$, the process $\widetilde W_n^*$ converges in $D(\R)$ in distribution conditionally to the process $V$
\end{lemma}

\begin{figure}[!ht]
\begin{subfigure}[b]{0.3\textwidth}
\includegraphics[width=\textwidth]{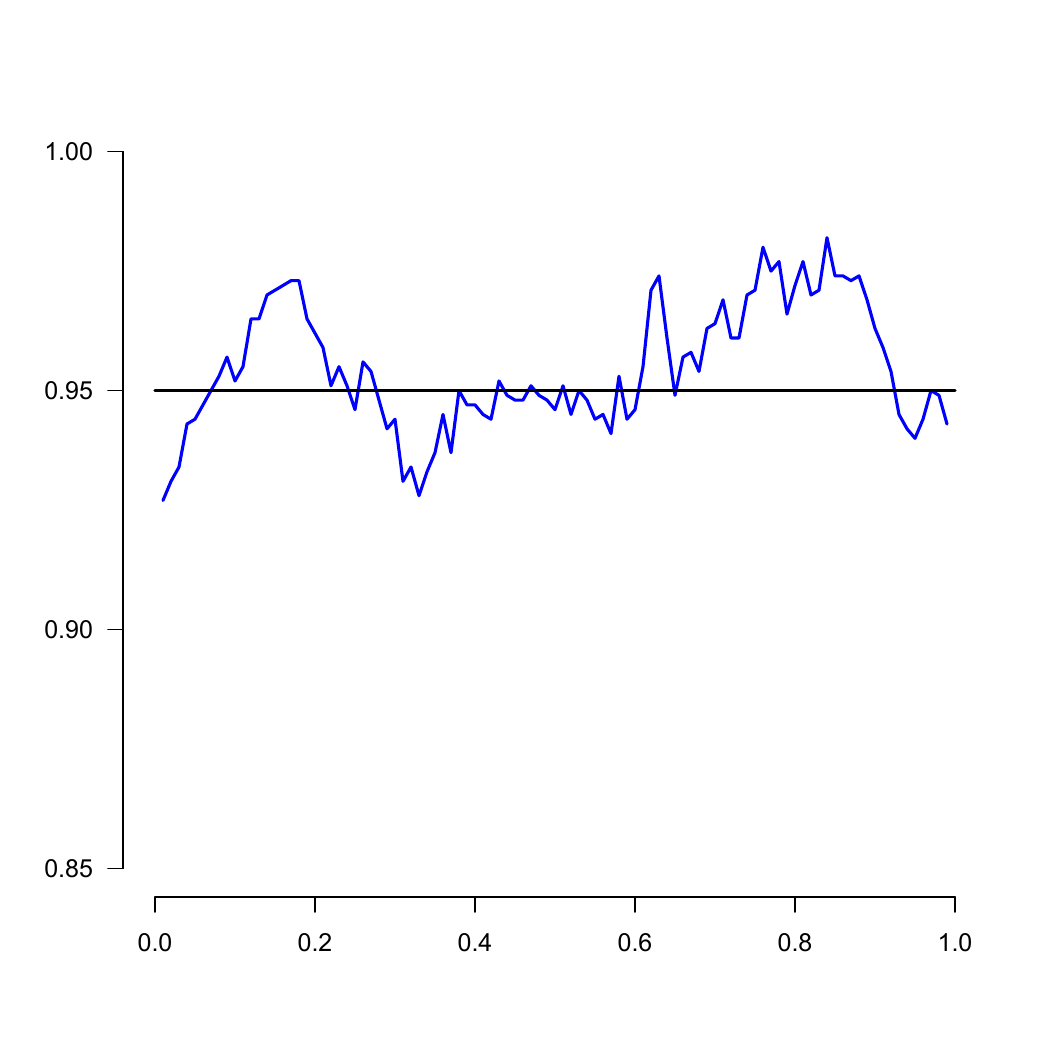}
\caption{}
\label{fig:p2}
\end{subfigure}
\begin{subfigure}[b]{0.3\textwidth}
\includegraphics[width=\textwidth]{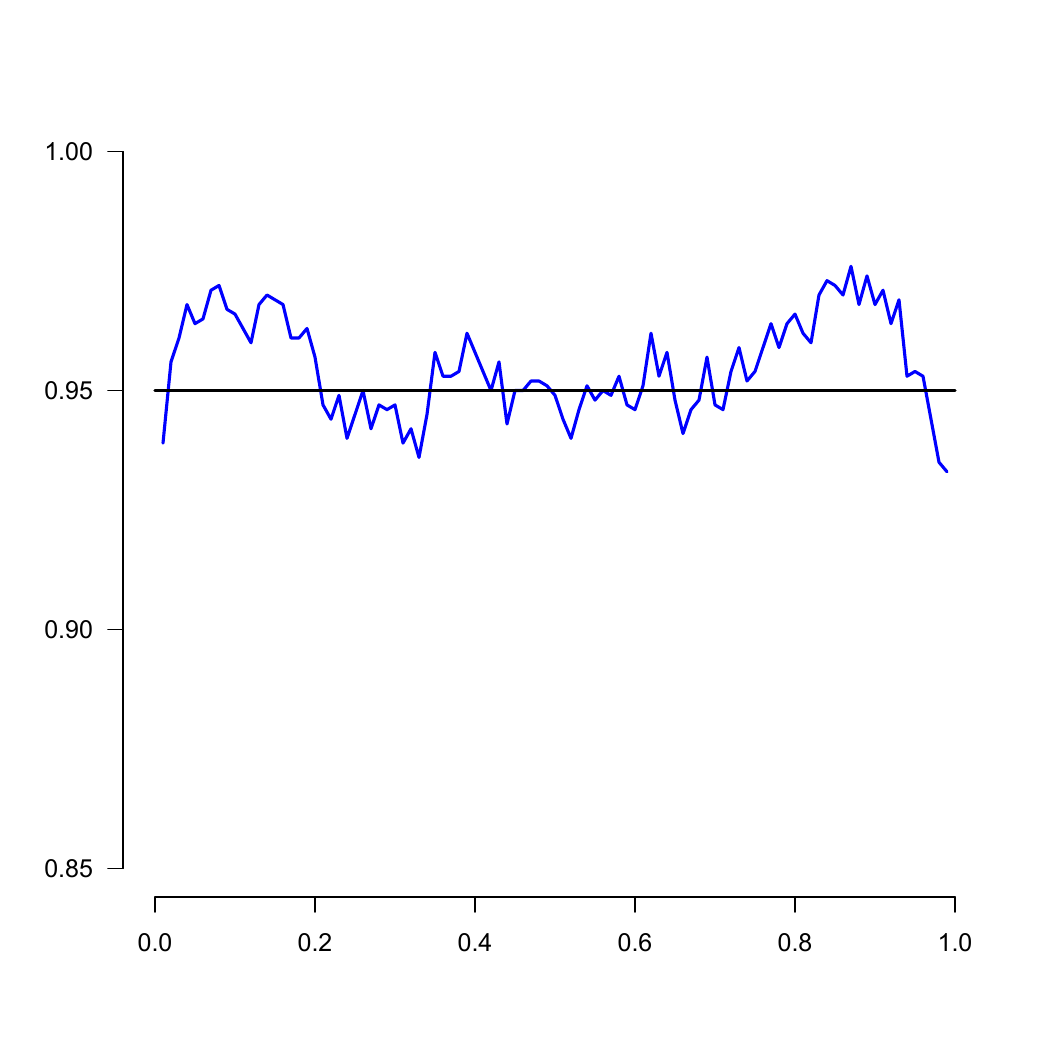}
\caption{}
\label{fig:q2}
\end{subfigure}
\begin{subfigure}[b]{0.3\textwidth}
\includegraphics[width=\textwidth]{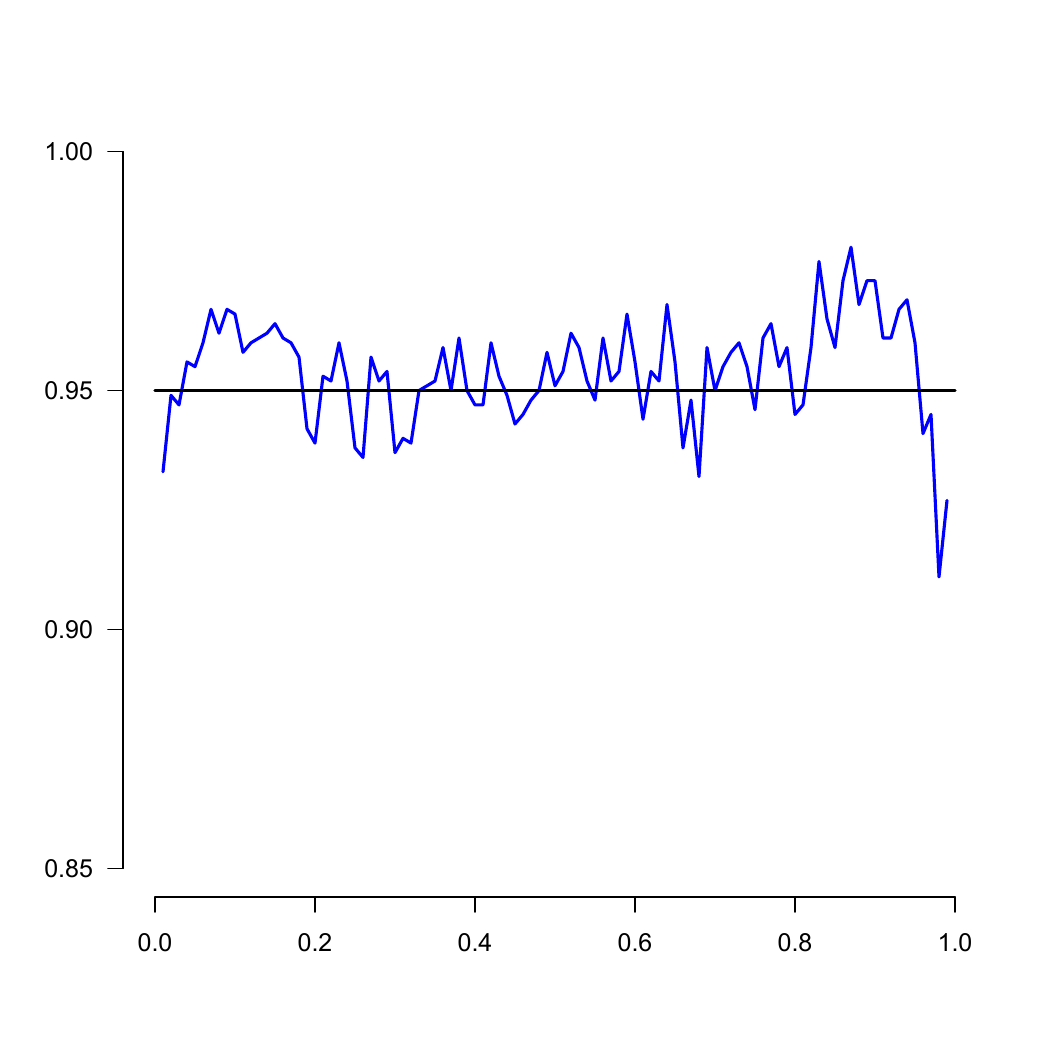}
\caption{}
\label{fig:r2}
\end{subfigure}
\caption{Coverage percentages for confidence intervals, based on (\ref{first_conf_intervals}), for $n=100, 500$ and $1000$.  (a) $n=100$, (b) $n=500$, (c) $n=1000$.}
\label{figure:CI_percentages_Chernoff}
\end{figure}

This leads to the following corollary.

\begin{corollary}
\label{cor:smoothed_bootstrap}
Let the bootstrap LSE $f_n^*$ be constructed as in (\ref{SLSE_intervals0}) and let $t_0$ and the SLSE $\tilde f_{nh}$ be defined as in Theorem \ref{th:limit_SLSE}. Then, along almost all sequences $(X_1,Y_1),\dots$, we have, under the condtions of Theorem \ref{th:limit_SLSE},
\begin{align*}
n^{1/3}\bigl\{\hat f_n^*(t_0)-\tilde f_{nh}(t_0)\bigr\}\stackrel{{\mathcal D}}\longrightarrow Z,
\end{align*}
where
\begin{align*}
Z=\left(\frac{4\s_0^2f_0'(t_0)}{g(t_0)}\right)^{1/3}\text{\rm argmin}_{t\in\R}\left[W(t)+t^2\right],
\end{align*}
and $W$ is standard two-sided Brownian motion, originating from zero.
\end{corollary}

\begin{proof}
We use the ``switch relation'' again (see the proof of Theorem \ref{theorem:limit_posterior}). Let $G_n$ be the empirical distribution function of the $X_i$ and let $V_n^*$ be defined by
\begin{align*}
V_n^*(t)=\sum_{X_i\le t}Y_i^*,
\end{align*}
where $Y_i^*$ is defined by (\ref{bootstrap_sample}). Let $U_n^*$ be defined by
\begin{align*}
U_n^*(a)=\text{argmin}\{t\in[0,1]:V_n^*(t)-a G_n(t)\},
\end{align*}
for $a$ in the range of $f_0$.
Then we have the ``switch relation'':
\begin{align*}
\hat f_n^*(t)\ge a \iff G_n(t)\ge G_n(U_n^*(a)) \iff t\ge U_n^*(a).
\end{align*}
Hence we get if $a_n=\tilde f_{nh}(t_0)$ and $D_n=\{(X_1,Y_1),\dots,(X_n,Y_n)\}$,
\begin{align*}
&\P\left\{n^{1/3}\{\hat f_n^*(t_0)-\tilde f_{nh}(t_0)\}\ge x|D_n\right\}=
\P\left\{\hat f_n^*(t_0)\ge a_n+n^{-1/3}x|D_n\right\}\\
&=\P\left\{U_n^*(a_n+n^{-1/3}x)\le t_0|D_n\right\}=\P\left\{n^{1/3}\bigl\{U_n^*(a_n+n^{-1/3}x)-t_0\bigr\}\le0|D_n\right\}\nonumber\\
&=\P\left\{\text{argmin}\left[t\in[0,1]:V_n^*(t)-(a_n+n^{-1/3}x)\,G_n(t)\right]\le 0|D_n\right\}\\
&=\P\left\{\text{argmin}\left[t\in[0,1]:V_n^*(t)-V_n^*(t_0)-(a_n+n^{-1/3}x)\{G_n(t)-G_n(t_0)\}\right]\le 0|D_n\right\},
\end{align*}
where the last equality holds since the values of the argmin function do not change if we add constants to the function for which we determine the argmin.

We have:
\begin{align*}
&\P\left\{\text{argmin}\left[t\in[0,1]:V_n^*(t)-V_n^*(t_0)-(a_n+n^{-1/3}x)\{G_n(t)-G_n(t_0)\}\right]\le 0|D_n\right\}\\
&=\P\left\{\text{argmin}_t\left[W_n^*(t)+n^{1/3}\int_{(t_0,t_0+n^{-1/3}t]}\tilde f_{nh}(u)\,dG_n(u)\right.\right.\nonumber\\
&\qquad\qquad\qquad\qquad\qquad\left.\left.-n^{1/3}(\tilde f_{nh}(t_0)+n^{-1/3}x)\{G_n(t_0+n^{-1/3})-G_n(t_0)\}\right]\le 0|D_n\right\}\\
&\sim
\P\left\{\text{argmin}_t\left[W_n^*(t)+\tfrac12f_{nh}'(t_0)g_0(t_0)t^2-xg(t_0)t\right]\le0|D_n\right\}\\
&\sim\P\left\{\text{argmin}_t\left[W_n^*(t)+\tfrac12f_0'(t_0)g_0(t_0)t^2-xg(t_0)t\right]\le0|D_n\right\}\\
&\longrightarrow \P\left\{\text{argmin}_t\left[\s_0\sqrt{g(t_0)}\,W(t)+\tfrac12f_0'(t_0)g(t_0)t^2-xg(t_0)t\right]\le0\right\},
\end{align*}
where we use Lemma \ref{lemma:localBM_convergence} in the last step. By Brownian scaling, we can write
\begin{align*}
&\P\left\{\text{argmin}_t\left[\s_0\sqrt{g(t_0)}\,W(t)+\tfrac12f_0'(t_0)g_0(t_0)t^2-xg(t_0)t\right]\le0\right\}\\
&=\P\left\{\left(\frac{4\s_0^2f_0'(t_0)}{g(t_0)}\right)^{1/3}\text{argmin}_t\left[W(t)+t^2\right]\le -x\right\}\\
&=\P\left\{\left(\frac{4\s_0^2f_0'(t_0)}{g(t_0)}\right)^{1/3}\text{argmin}_t\left[W(t)+t^2\right]\ge x\right\}.
\end{align*}
\end{proof}

\begin{remark}
{\rm
Corollary \ref{cor:smoothed_bootstrap} shows that the limit distribution of $n^{1/3}\bigl\{\hat f_n^*(t_0)-\tilde f_{nh}(t_0)\bigr\}$, along almost all sequences $D_n$, is the same as the limit distributionof $n^{1/3}\bigl\{\hat f_n(t_0)-\tilde f_0(t_0)\bigr\}$. In the latter case the limit distribution was first derived in \cite{brunk:70}.
}
\end{remark}

Using this corollary it is clear that in using this method the (ordinary, not percentile) smoothed bootstrap simulations recreate the actual asymptotic distribution correctly, and that we do not have to use a correction for over- or undercoverage. It is also clear from Figure \ref{figure:CI_percentages_Chernoff} that its behavior is much better than the behavior of the confidence intervals in the preceding section. Even for sample size $n=100$ the confidence intervals are more or less ``on target''. In comparison, the credible intervals are still far off the target for these sample sizes, and the overcoverage has still not set in for these sample sizes, see Figure \ref{figure:CI_percentages_credible}.

\begin{figure}[!ht]
\begin{subfigure}[b]{0.3\textwidth}
\includegraphics[width=\textwidth]{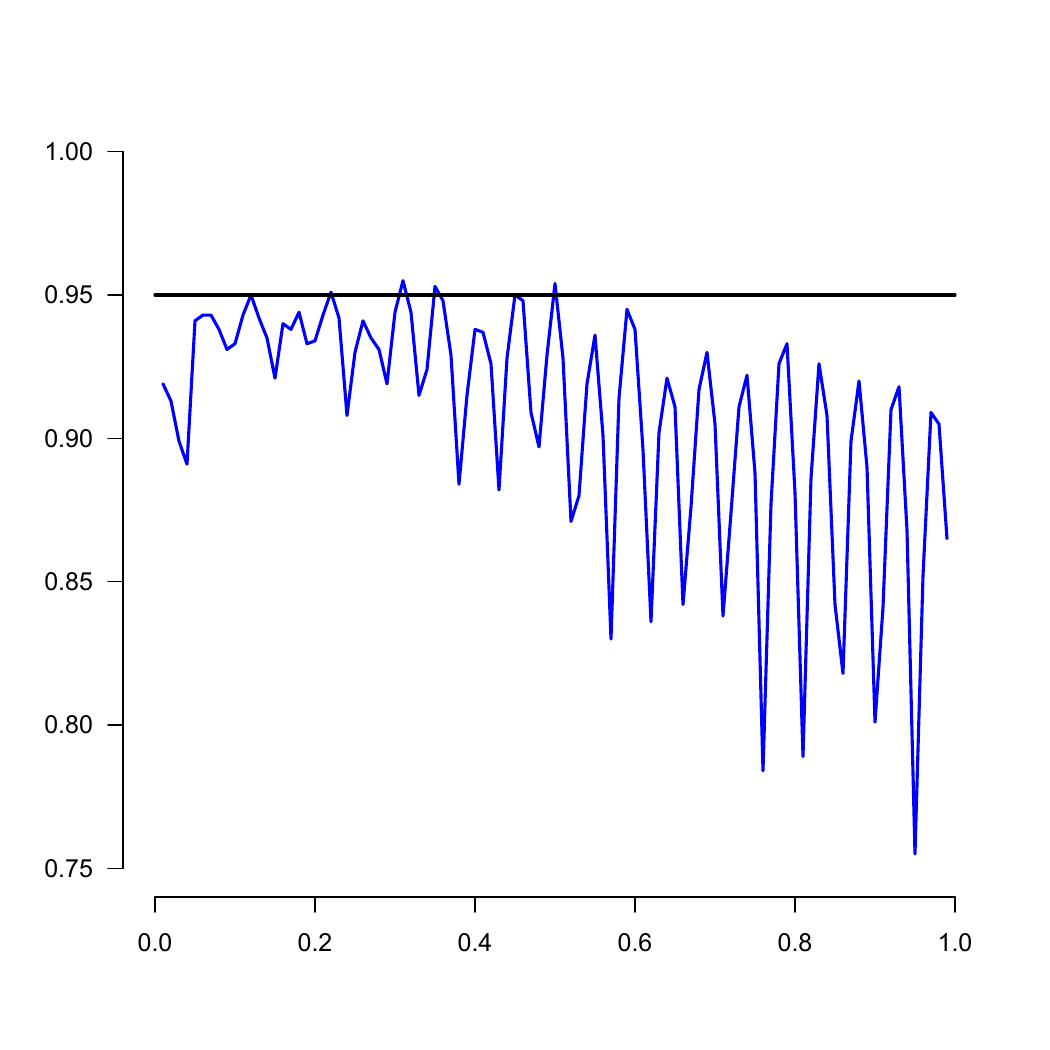}
\caption{}
\label{fig:percentages_credible100}
\end{subfigure}
\begin{subfigure}[b]{0.3\textwidth}
\includegraphics[width=\textwidth]{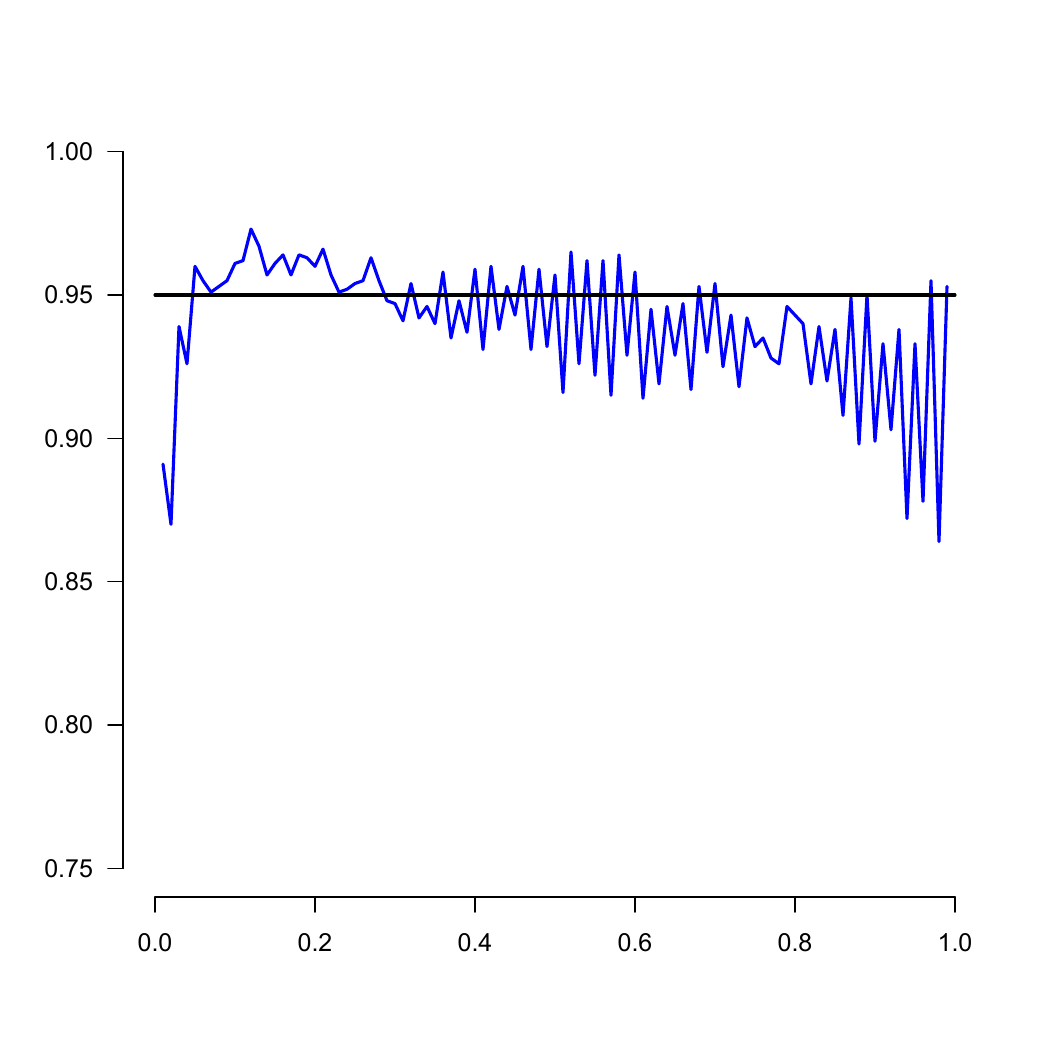}
\caption{}
\label{fig:percentages_credible500}
\end{subfigure}
\begin{subfigure}[b]{0.3\textwidth}
\includegraphics[width=\textwidth]{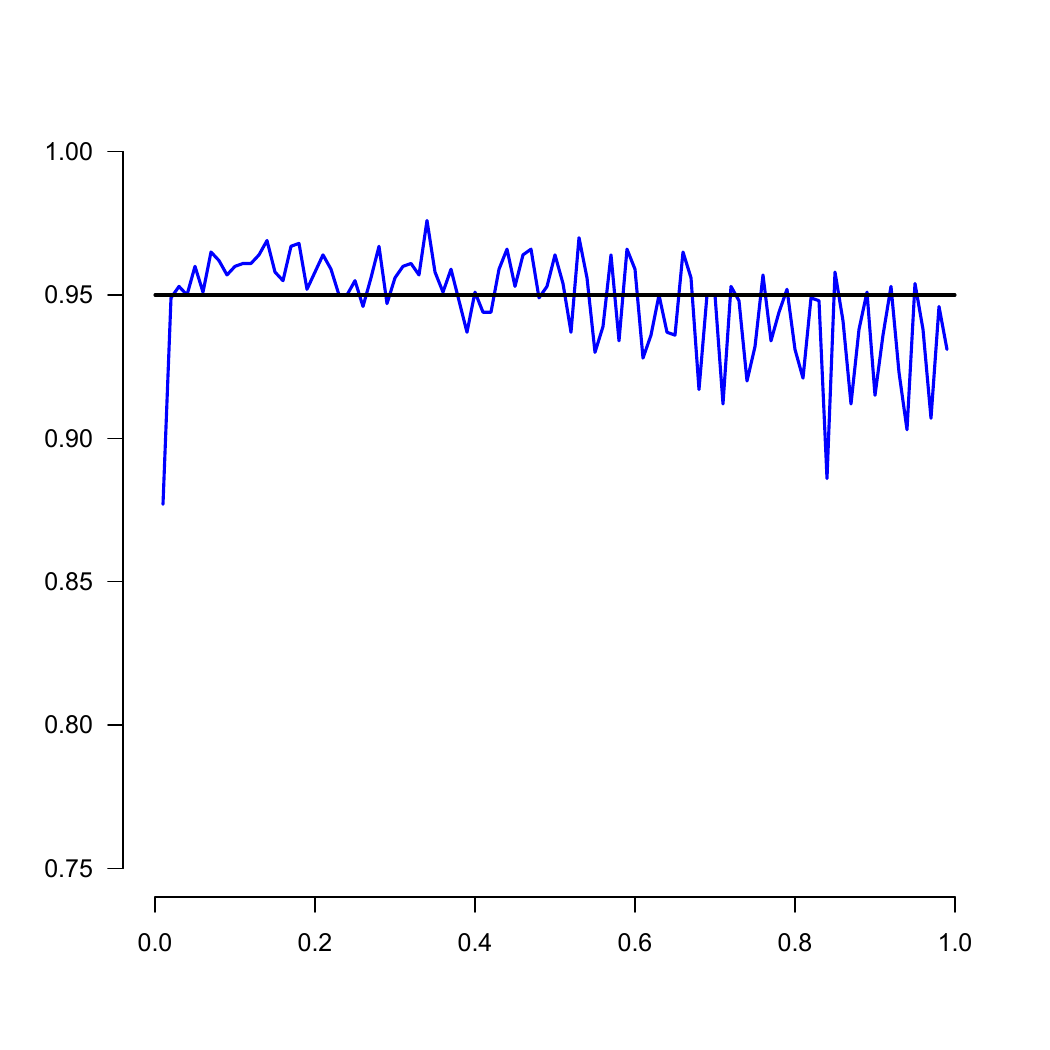}
\caption{}
\label{fig:percentages_credible1000}
\end{subfigure}
\caption{Coverage percentages for the confidence intervals,  for   (a) $n=100$, (b) $n=500$, (c) $n=1000$.}
\label{figure:CI_percentages_credible}
\end{figure}

We can also prove the result below, illustrating the fact that there is no need for correction for over-or under coverage.
\begin{theorem}
\label{theorem:limit_smooth_bootstrap}
Let $D_n=\{(X_1,Y_1),\dots,(X_n,Y_n)\}$ and $h\asymp n^{-1/5}$, and let $t_0$ and $\tilde f_{nh}$ be defined as in Corollary \ref{cor:smoothed_bootstrap}. Let the conditions of Theorem \ref{th:limit_SLSE} be satisfied. Then, for $z\in(0,1)$, as $n\to\infty$,
\begin{align}
\label{smoothed_bootrstrap_values}
&\P\left\{\P\left(n^{1/3}\bigl\{\hat f_n^*(t_0)-\tilde f_{nh}(t_0)+\hat f_n(t_0)-f_0(t_0)\bigr\}\le 0\Bigm|D_n\right)\le z\right\}\longrightarrow z.
\end{align}
\end{theorem}

\begin{proof}
Let
\begin{align*}
U_n=n^{1/3}\bigl\{\hat f_n^*(t_0)-\tilde f_{nh}(t_0)\bigr\},\qquad V_n=n^{1/3}\bigl\{\hat f_n(t_0)-f_0(t_0)\bigr\},
\end{align*}
and let $\Phi$ be the distribution function of $Z$, defined in Corollary \ref{cor:smoothed_bootstrap}.
We have
\begin{align*}
&\P\left\{\P\left(n^{1/3}\bigl\{\hat f_n^*(t_0)-\tilde f_{nh}(t_0)+\hat f_n(t_0)-f_0(t_0)\bigr\}\le 0\Bigm|D_n\right)\le z\right\}\\
&=\P\left\{\P\{V_n\le -U_n|D_n\}\le z\right\}\\
&\sim\P\left\{\Phi(-U_n)\}\le z\right\}=\P\left\{-U_n\le \Phi^{-1}(z)\right\}
\longrightarrow \Phi\circ \Phi^{-1}(z)=z,
\end{align*}
using the symmetry around zero of the distribution of $Z$ (the limit in distribution of $U_n$).
\end{proof}

\begin{figure}[!ht]
\centering
\includegraphics[width=0.45\textwidth]{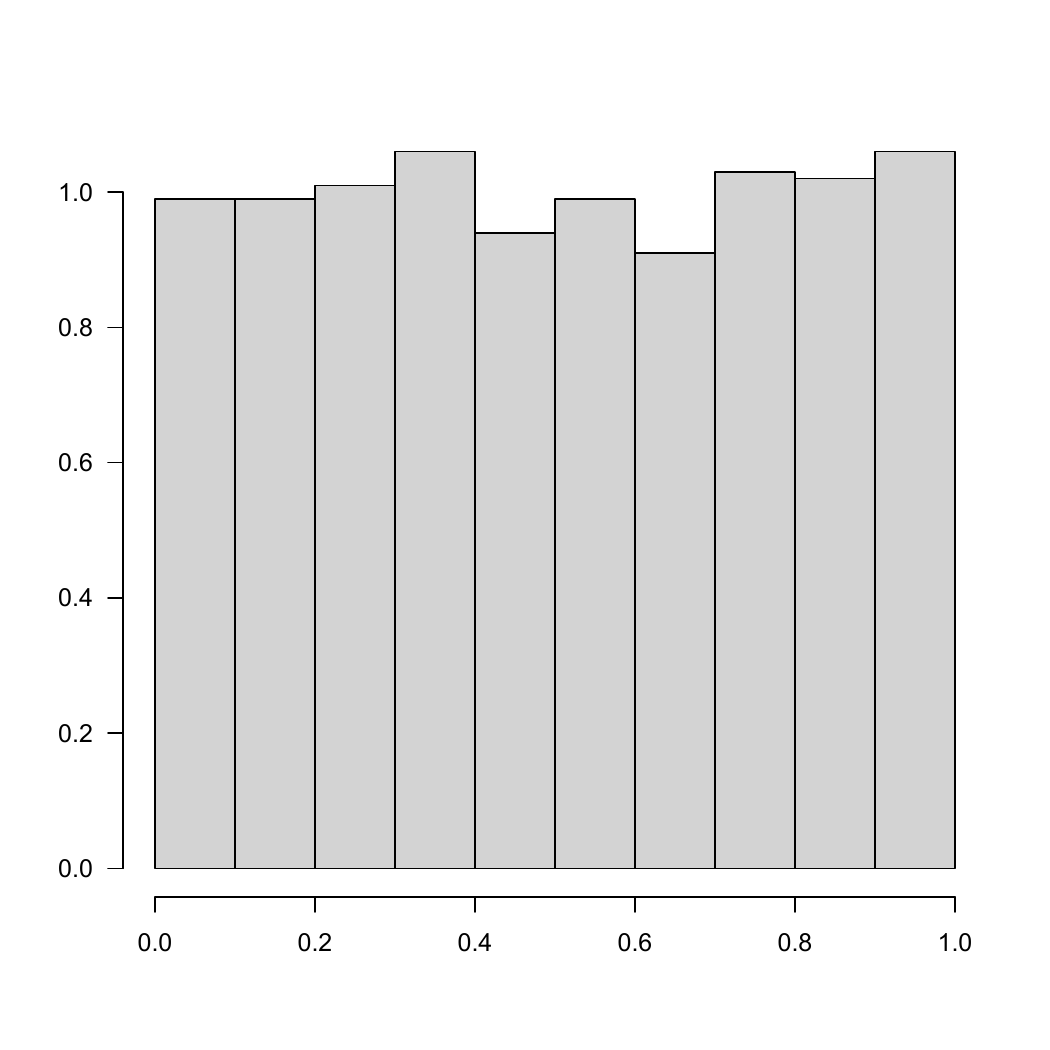}
\caption{Histogram of $1000$ estimates of $\P\left(n^{1/3}\bigl\{\hat f_n^*(t_0)-\tilde f_{nh}(t_0)+\hat f_n(t_0)-f_0(t_0)\bigr\}\le 0\bigm|D_n\right)$ for the smoothed bootstrap, for $1000$ samples $D_n$ of size $n=2000$ and $t_0=0.5$. }
\label{figure:histogram_smoothed_bootstrap2000}
\end{figure}

Note that $\hat f_n^*(t_0)$ is now centered by $\tilde f_{nh}(t_0)$ instead of $\hat f_n$, and that $n^{1/3}\{\hat f_n^*(t_0)-\tilde f_{n,h}(t_0)\}$ tends to the right limit distribution, in contrast with $n^{1/3}\{\hat f_n^*(t_0)-\hat f_n(t_0)\}$. The histogram of estimates of $\P\bigl(n^{1/3}\{\hat f_n^*(t_0)-\tilde f_{nh}(t_0)+\hat f_n(t_0)-f_0(t_0)\}\le 0\bigm|D_n\bigr)$, based on relative frequencies, is shown in Figure \ref{figure:histogram_smoothed_bootstrap2000}.

For the ordinary (non-percentile) bootstrap we get by an entirely similar proof, in which we do not need the symmetry of the limit distribution:

\begin{theorem}
\label{theorem:limit_smooth_bootstrap2}
Let $D_n=\{((X_1,Y_1),\dots,(X_n,Y_n)\}$ and $h\asymp n^{-1/5}$. Let the conditions of Theorem \ref{th:limit_SLSE} be satisfied. Then, for $z\in(0,1)$, as $n\to\infty$,
\begin{align}
\label{smoothed_bootrstrap_values2}
&\P\left\{\P\left(\hat f_n^*(t_0)-\tilde f_{nh}(t_0)\le \hat f_n(t_0)-f_0(t_0)\Bigm|D_n\right)\le z\right\}\longrightarrow z.
\end{align}
\end{theorem}

The result gives an interesting consequence of what it means to say that the "bootstrap works''. This phenomenon also occurs in the simple bootstrap setting where one resamples with replacement from samples $U_1,\ldots,U_n$ from a normal $N(\mu,\sigma^2)$ distribution with the aim to construct a confidence set for the mean. Then also, for all $z\in(0,1)$, 
\begin{align}
	\P\left\{\P\left(\bar{U}_n^*-\bar{U}_n\le \bar{U}_n-\mu\Bigm|U_1,\ldots U_n\right)\le z\right\}\longrightarrow z.
\end{align}

\section{Concluding remarks}
We showed that the construction of pointwise credible intervals for the monotone regression function, as proposed in \cite{moumita:21}, has an interpretation as the construction of percentile bootstrap intervals. The overcoverage,
as explained by Theorem \ref{theorem:limit_posterior}, only sets in for very large sample sizes, like $n=20,000$; for smaller sample sizes we have observed undercoverage.

Because the confidence intervals are based on piecewise constant estimates of the regression function, on intervals of equal length, the effect of bias is very pronounced, which does not hold for the confidence intervals, based on the smoothed bootstrap in Section \ref{sec:CI_cuberoot_n}. The latter confidence intervals have the further advantages of being on target also for smaller sample sizes, and not needing correction for overcoverage or undercoverage, since the estimates are consistent.

The consistency is also borne out by Theorem \ref{theorem:limit_smooth_bootstrap2}, showing convergence to the uniform distribution of $\P(\hat f_n^*(t_0)-\tilde f_{nh}(t_0)\le \hat f_n(t_0)-f_0(t_0)|D_n)$, $D_n=\{(X_1,Y_1),\dots,(X_n,Y_n)\}$, for the smoothed bootstrap estimates $\hat f_n^*$, in contrast with the situation for the credible intervals and the percentile bootstrap intervals, where we need a correction for convergence of $\P(\hat f_n^*(t_0)-f_0(t_0)\le 0\bigm|D_n)$ to a distribution different from the uniform distribution.

As shown in \cite{geurt_piet:23}, it is also possible to use the smoothed least squares estimator (SLSE) directly as the basis for confidence intervals. In this case,  resampling is done from residuals w.r.t.\ an oversmoothed estimate of the regression function to treat the bias in the right way. The bias is in this case much more of a problem because the variance and squared bias of the SLSE are of the same order if the bandwidth is of order $n^{-1/5}$. A picture of confidence intervals of this type is given in Figure \ref{figure:CI_first_intervals100}. For more details, see \cite{geurt_piet:23}.

All simulations in our paper can be recreated using the ${\tt R}$ scripts in \cite{piet_github:21}.

\begin{figure}[!ht]
\includegraphics[width=0.5\textwidth]{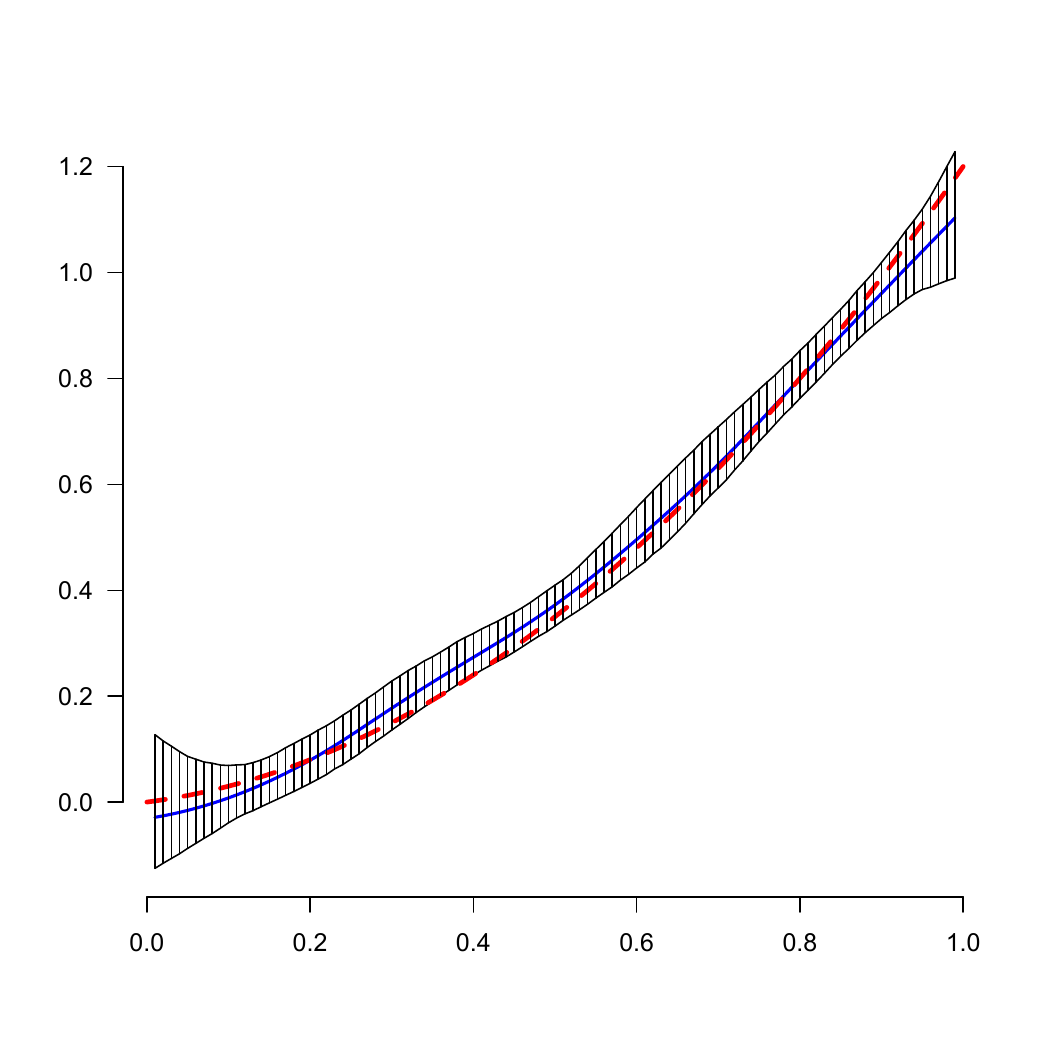}
\caption{The SLSE (blue, solid) and $95\%$ confidence intervals, using the theory in \cite{geurt_piet:23}, for sample size $n=100$ and $f_0(x)=x^2+x/5$ (red dashed curve).}
\label{figure:CI_first_intervals100}
\end{figure}

\section*{Appendix}
{\sc Derivation of raw posterior distribution (\ref{posterior_theta})}\\
Considering $\bm X$ fixed, we have $\bm Y=\bm B\bm\theta+\bm \epsilon$, where $\bm B=(b_{ij})$ with $b_{ij}=1_{I_j}(X_i)$ giving $\bm Y|\bm\theta\sim N_n(\bm B\bm\theta,\sigma_0^2I_n)$. Writing $\bm \Lambda={\rm Diag}(\lambda_j^2)$, this can be combined with the prior $\bm \theta\sim N_J(\bm\zeta,\sigma_0^2\Lambda)$ yielding
\begin{align*}
f(\bm\theta|\bm Y)\propto f(\bm Y|\bm\theta)f(\bm\theta)\propto \exp\left(-\frac1{2\sigma_0^2}\left[\bm\theta^T(\bf\Lambda^{-1}+\bm B^T\bm B)\bm\theta-2\bm\theta^T(\bm B^T\bm Y+\bm \Lambda^{-1}\bm\zeta)\right]\right)
\end{align*}  
By `completing the square', this function can be seen to be proportional to the normal density with covariance matrix
$$
\sigma_0^2\left(\bm\Lambda^{-1}+\bm B^T\bm B\right)^{-1}=\sigma_0^2{\rm Diag}(1/\lambda_j^2+n_j)^{-1}={\rm Diag}\left(\frac{\sigma_0^2}{1/\lambda_j^2+n_j}\right)
$$
where we use that both $\bm \Lambda$ and $\bm B^T\bm B$ are diagonal matrices with diagonal entries $\lambda_j^2$ and $n_j=\sum_i b_{ij}=\#\{i\,:x_i\in I_j\}$ respectively. The expectation is given by
$$
\left( \bm\Lambda^{-1}+\bm B^T\bm B \right)^{-1}(\bm B^T\bm Y+\bm\Lambda^{-1}\bm\zeta)={\rm Diag}\left(\frac{1}{1/\lambda_j^2+n_j}\right)
\left(\begin{array}{c}\zeta_1/\lambda_1^2+\sum_{\{i\,:\,x_i\in I_1\}} y_i\\\vdots\\\zeta_J/\lambda_J^2+\sum_{\{i\,:\,x_i\in I_J\}}y_i\end{array}\right),
$$ 
boiling down to the expression in (\ref{posterior_theta}).\\

\vspace{0.5cm}
\noindent
{\sc Derivation of  (\ref{eq:exprEmpBayesSigma2})}\\
First note that  $I+B\Lambda B^T$ is a block diagonal matrix, where block $j$,  has size $n_j\times n_j$, diagonal elements $1+\lambda_j^2$ and off-diagonal elements $\lambda_j^2$ for $1\le j\le J$. The $j$-th block can be written as
$$
A_j=I_{n_j\times n_j}+\lambda_j^2\one_{n_j} \one_{n_j}^T,
$$
where $I_{k\times k}$ is the $k\times k$ identity matrix and $\one_k$ is the column vector of length $k$ with all elements equal to one. 

This means that $(I+B\Lambda B^T)^{-1}$ is also a block diagonal matrix, with $j$-th block
$$
A_j^{-1}=I_{n_j\times n_j}-\frac{\lambda_j^2}{1+n_j\lambda_j^2}\one_{n_j} \one_{n_j}^T,
$$
by the Sherman Morrison formula.
For convenience, write $z=y-B\zeta$ and denote by subscript $[j]$ the part of a vector corresponding to the $j$-th block (so $i$ for which $x_i\in I_j$; $z_{[j]}$ has length $n_j$). Then
$$
\left[(I+B\Lambda B^T)^{-1}z\right]_{[j]}=A_j^{-1}z_{[j]}=z_{[j]}-\frac{\lambda_j^2}{1+n_j\lambda_j^2}\one_{n_j} \one_{n_j}^Tz_{[j]}
=z_{[j]}-\frac{\lambda_j^2n_j\bar{z}_{[j]}}{1+n_j\lambda_j^2}\one_{n_j},
$$
where $\bar{z}_{[j]}$ denotes the average of the entries of $z_{[j]}$. Therefore
$$
z_{[j]}^T\left[(I+B\Lambda B^T)^{-1}z\right]_{[j]}=z_{[j]}^Tz_{[j]}-\frac{\lambda_j^2n_j^2\bar{z}_{[j]}^2}{1+n_j\lambda_j^2}.
$$
Hence,
\begin{align}
	\label{eq:exprsigmsq}
z^T(I+B\Lambda B^T)^{-1}z=\sum_{j=1}^Jz_{[j]}^T\left[(I+B\Lambda B^T)^{-1}z\right]_{[j]}
=z^Tz -\sum_{j=1}^J\frac{n_j\bar{z}_{[j]}^2}{1+1/(n_j\lambda_j^2)}.
\end{align}
Now write $(1+1/(n_j\lambda_j^2))^{-1}=1-\delta_j$, so $\delta_j=(1+n_j\lambda_j^2)^{-1}$. Then (\ref{eq:exprsigmsq}) can be further rewritten as
\begin{align}
	\label{eq:uitwerking}
&\,\,\,z^T(I+B\Lambda B^T)^{-1}z=z^Tz -\sum_{j=1}^J(1-\delta_j)n_j\bar{z}_{[j]}^2=z^Tz -\sum_{j=1}^Jn_j\bar{z}_{[j]}^2+\sum_{j=1}^J\delta_jn_j\bar{z}_{[j]}^2\\
&\,\,\,=\sum_{j=1}^J\sum_{x_i\in I_j}\left((y_i-\zeta_j)^2-(\bar{y}_j-\zeta_j)^2\right)+\sum_{j=1}^J\delta_jn_j(\bar{y}_j-\zeta_j)^2\nonumber\\
&\,\,\,=\sum_{j=1}^J\sum_{x_i\in I_j}\left(y_i^2-\bar{y}_j^2\right)+\sum_{j=1}^J\delta_jn_j(\bar{y}_j-\zeta_j)^2=
\sum_{j=1}^J\sum_{x_i\in I_j}\left(y_i-\bar{y}_j\right)^2+\sum_{j=1}^J\delta_jn_j(\bar{y}_j-\zeta_j)^2.\nonumber
\end{align}
Substituting $\delta_j=(1+n_j\lambda_j^2)^{-1}$ and writing the first term as one sum, yields (\ref{eq:exprEmpBayesSigma2}).\\

\vspace{0.5cm}
\noindent
{\sc Derivation of Empirical Bayes estimators (\ref{eq:empBayesZeta}) and (\ref{eq:empBayesSigma2})}
The distribution of the observed $\bm Y$ can be expressed in terms of the parameters $\sigma_0^2$, $\Lambda$ and $\zeta$,
$$
{\bm Y}={\bm B}{\bm \theta}+\epsilon={\bm B}({\bm \zeta}+\sigma_0\tilde{\epsilon})+\sigma_0\epsilon={\bm B}\zeta+\sigma_0({\bm B}\tilde{\epsilon}+\epsilon),
$$
where $\epsilon\sim N_n(0,I_{n\times n})$ and $\tilde{\epsilon}\sim N_J(0,\Lambda)$ are independent. Therefore, 
$$
{\bm Y}\sim N_n\left({\bm B}\zeta,\sigma_0^2\left(I_{n\times n}+{\bm B}{\bm \Lambda}{\bm B}^T\right)\right).
$$
Maximizing the likelihood in $\zeta$, for fixed values of $\sigma_0^2$ and $\Lambda$ entails minimizing
$$
(y-B\zeta)^T\left(I_{n\times n}+{\bm B}{\bm \Lambda}{\bm B}^T\right)^{-1}(y-B\zeta).
$$
Recognizing (\ref{eq:uitwerking}) in this expression, it is clear that the empirical Bayes estimate of $\zeta$ is either given by the vector $(\bar{y}_1,\ldots,\bar{y}_J)^T$ if the likelihood is maximized over $\R^J$ or its isotonic regression with weights $n_j\delta_j=n_j/(1+n_j\lambda_j^2)$ if monotonicity is taken into account. 

For any fixed value of $\zeta$, maximizing the the log likelihood of $\sigma_0$ corresponds to minimizing
$$
\frac{n}{2}\log\sigma_0^2+(y-B\zeta)^T\left(I_{n\times n}+{\bm B}{\bm \Lambda}{\bm B}^T\right)^{-1}(y-B\zeta)/(2\sigma_0^2),
$$
yielding (\ref{eq:empBayesSigma2}).

\vspace{0.5cm}
\noindent
\begin{proof}[Proof of Lemma \ref{lemma:localBM_convergence_classic}]
We employ a construction, used in the proof of Lemma 2.2 in \cite{piet:88}. Let $A_n$ be the interval $[t_0-n^{-1/3}\log n,t_0+n^{-1/3}\log n]$ and let $(U_1^*,V_1^*),(U_2^*,V_2^*),\dots$ be an i.i.d sequence of points, (discretely) uniformly distributed on the set of points $(X_i,Y_i)$ such that $X_i\in A_n$.
 
Let $M_n$ be the number of points $X_i\in A_n$. The number of bootstrap draws  such that the first component belongs to $A_n$ has distribution
\begin{align}
\label{restricted_bootstrap}
M_n^*\sim\text{Binom}(n,M_n/n),
\end{align}
so, taking the random variable $M_n^*$ defined by (\ref{restricted_bootstrap}), independent of the sequence $(U_1^*,V_1^*)$, $(U_2^*,V_2^*),\dots$, we can represent the bootstrap variables  such that the first component belongs to $A_n$ by
\begin{align*}
\sum_{i=1}^{M_n^*}\d_{\{(U_i^*,V_i^*)\}},
\end{align*}
where $\d_x$ denotes Dirac measure.

We can couple this process with a Poisson process
\begin{align*}
\sum_{i=1}^{N_n}\d_{\{(U_i^*,V_i^*)\}},
\end{align*}
where
\begin{align*}
N_n\sim \text{Poisson}\,(M_n),
\end{align*}
independent of the $(U_i^*,V_i^*)$, using the construction with a Uniform(0,1) random variable $U$ as in the construction in the proof of Lemma 2.2 in \cite{piet:88}. We find in this way:
\begin{align*}
\P\left\{\sum_{i=1}^{M_n^*}\d_{\{(U_i^*,V_i^*)\}}\ne \sum_{i=1}^{N_n}\d_{\{(U_i^*,V_i^*)\}}\right\}\le 2M_n/n,
\end{align*}
where $M_n/n$ tends to zero almost surely, using an inequality from \cite{vervaat:69}.

This means that we can replace $M_{in}^*$ by $N_{in}$ in (\ref{classic_W_n^*}), where the $N_{in}$ are independent Poisson$(1)$ random variables and can replace $\widetilde W_n^*$ by its Poissonized version
\begin{align}
\label{classic_Poisson W_n^*}
W_n^{(P)}(t)=
&n^{-1/3}\sum_{j:j/J\in[0,t_0+n^{-1/3}t]}\sum_{X_i\in I_j}(N_{in}-1)(Y_i-a_0-n^{-1/3}x)\\
&\qquad\qquad\qquad\qquad-n^{-1/3}\sum_{j:j/J\in[0,t_0]}\sum_{X_i\in I_j}(N_{in}-1)(Y_i-a_0-n^{-1/3}x).
\end{align}
For the latter process we have the martingale structure again, and the quadratic variation process $[W_n^{(P)}](t),\,t\ge0$ satisfies
\begin{align*}
 &\left[W_n^{(P)}\right](t)=n^{-2/3}\sum_{j:j/J\in(t_0,t_0+n^{-1/3}t]}\Bigl\{\sum_{X_i\in I_j}(N_{in}-1)(Y_i-a_0-n^{-1/3}x)\Bigr\}^2\\
&\longrightarrow \s_0^2g(t_0)t.
\end{align*}
for almost all sequences $(X_1,Y_1),\dots$. A similar reation holds for $t<0$.
So the result follows in the same way as in the proof of Lemma \ref{lemma:localBM_convergence_theta}.
\end{proof}

\vspace{0.5cm}
\noindent
\begin{proof}[Proof of Lemma \ref{lemma:localBM_convergence}]
 We consider the case $t\ge0$. It is clear that, conditionally on $(X_1,Y_1),\dots,(X_n,Y_n)$,  $t\mapsto \widetilde W_n^*(t)$ is a martingale with respect to the family of $\s$-algebras $ {\cal F}^*_{n,t},\,t\ge0$, defined by:
 \begin{align*}
 {\cal F}^*_{n,t}=\s\left\{(X_i,E_i^*):X_i\in[t_0+n^{-1/3}t]\right\}.\qquad t\ge0.
 \end{align*}
 The quadratic variation process is given by:
\begin{align*}
 \left[\widetilde W_n^*\right](t)=n^{-2/3}\sum_{i:X_i\in[t_0,t_0+n^{-1/3}t]}\left(E_i^*\right)^2
\end{align*}
We have:
\begin{align*}
&\E\left\{ \left[\widetilde W_n^*\right](t)\Bigm|(X_1,Y_1),\dots,(X_n,Y_n)\right\}\\
&=n^{-2/3}\sum_{i:X_i\in[t_0,t_0+n^{-1/3}t]}\E\left\{\left(E_i^*\right)^2\Bigm|(X_1,Y_1),\dots,(X_n,Y_n)\right\}\\
&=n^{-2/3}\sum_{i:X_i\in[t_0,t_0+n^{-1/3}t]}n^{-1}\sum_{j=1}^n \tilde E_j^2\stackrel{a.s.}\longrightarrow\s_0^2g(t_0)t.
\end{align*}
Note that
\begin{align*}
&\bigl\{Y_j-\tilde f_{nh}(X_j)\bigr\}^2\\
&=\bigl\{Y_j-f_0(X_j)\bigr\}^2+\bigl\{f_0(X_j)-\tilde f_{nh}(X_j)\bigr\}^2+2\bigl\{Y_j-f_0(X_j)\bigr\}\bigl\{f_0(X_j)-\tilde f_{nh}(X_j)\bigr\},
\end{align*}
and that therefore
\begin{align*}
n^{-1}\sum_{j=1}^n\tilde E_j^2\sim n^{-1}\sum_{j=1}^n\bigl\{Y_j-f_0(X_j)\bigr\}^2
\end{align*}
almost surely, as $n\to\infty$, by the properties of $\tilde f_{nh}$.

We can treat the case $t<0$ in a similar way.
This means that we can apply Rebolledo's theorem, see Theorem 3.6, p.\ 68 of \cite{piet_geurt:14} and \cite{rebolledo:80}. The conclusion of the lemma now follows.
 \end{proof}

\bibliographystyle{imsart-nameyear}
\bibliography{cupbook}

\end{document}

%% file: commands3.tex
\def\a{\alpha}
\def\b{\beta}
\def\th{\theta}
\def\R{\mathbb R}
\def\dd{\Delta}

\def\d{\delta}

\def\E{{\mathbb E}}

\def\G{{\mathbb G}}

\def\P{{\mathbb P}}

\def\l{\lambda}
\def\labda1{\lambda_1}
\def\labda2{\lambda_2}

\def\e{\varepsilon}

\def\s{\sigma}

\def\comment#1{\relax}

\def\=in{\mathop{\rm =}}